\documentclass[10pt]{amsart}
\usepackage[margin=3cm]{geometry}
\usepackage{amsmath}
\usepackage{amssymb}
\usepackage{amsbsy}
\usepackage[latin1]{inputenc}
\usepackage{dsfont}
\usepackage{graphicx}
\usepackage{tcolorbox}
\usepackage{subfigure}
\usepackage[colorlinks=true,breaklinks=true,linkcolor=blue]{hyperref}
\usepackage{enumerate}

\def\a        {{\boldsymbol a}}
\def\b        {{\boldsymbol b}}
\def\e        {{\boldsymbol e}}
\def\x        {{\boldsymbol x}}
\def\y        {{\boldsymbol y}}

\def\n        {{\boldsymbol n}}

\def\dx     {{\rm d}{\boldsymbol x}}

\def\dt     {{\rm d}t}

\def\R {{\mathds R}}

\newtheorem{theorem}{Theorem}[section]
\newtheorem{lemma}[theorem]{Lemma}
\newtheorem{proposition}[theorem]{Proposition}
\newtheorem{corollary}[theorem]{Corollary}
\newtheorem{definition}[theorem]{Definition}

\begin{document}

\title[Numerical solution for a degenerate aggregation equation]{Numerical solution for an aggregation equation with degenerate diffusion}

\author[R. C. Cabrales]{Roberto Carlos Cabrales$^\mathsection$}
\address{$\mathsection$ Instituto de Investigaci\'on Multidisciplinaria 
en Ciencia y Tecnolog\'ia, Universidad de la Serena, 
Benavente 980, La Serena, Chile. E-mail: {\tt rcabrales@userena.cl}}
\thanks{RCC was partially supported by project PR18511 from Direcci\'on de 
Investigaci\'on of Universidad de La Serena}
\author[J. V. Guti\'errez-Santacreu]{Juan Vicente Guti\'errez-Santacreu$^\dag$}
\address{$\dag$ Dpto. de Matemática Aplicada I, E. T. S. I. Informática, Universidad de Sevilla. Avda. Reina Mercedes, s/n. E-41012 Sevilla, Spain.  E-mail: {\tt juanvi@us.es}. }
\thanks{JVGS and JRRG were partially supported by the Spanish Grant No. PGC2018-098308-B-I00 from Ministerio de Ciencias, Innovación y Universidades with the participation of FEDER}

\author[J. R. Rodríguez-Galván]{José Rafael Rodríguez-Galván$^\ddag$}
\address{$\ddag$ Departamento de Matemáticas. Facultad de Ciencias. Campus Universitario de Puerto Real, Universidad de Cádiz. 11510 Puerto Real, Cádiz E-mail: {\tt rafael.rodriguez@uca.es}}
%\thanks{JRRG}
\date{\today}

\begin{abstract} A numerical method for approximating weak solutions of an aggregation equation with degenerate diffusion is introduced. The numerical method consists of a stabilized finite element method together with a mass lumping technique and an extra stabilizing term plus a semi--implicit Euler time integration. Then we carry out a rigorous passage to the limit as the spatial and temporal discretization parameters tend to zero, and show that the sequence of finite element approximations converges toward the unique weak solution of the model at hands. In doing so, nonnegativity is attained due to the stabilizing term and the acuteness on partitions of the computational domain, and hence a priori energy estimates of finite element approximations are established. As we deal with a nonlinear problem, some form of strong convergence is required. The key compactness result is obtained via an adaptation of a Riesz--Fréchet--Kolmogorov criterion by perturbation. A numerical example is also presented.

\end{abstract}

\maketitle



\tableofcontents

\section{Introduction}
\subsection{The model}
 Let $\Omega\subset \R^d$, $d=2$ or $3$, be a bounded domain and $T>0$ be a fixed time. We consider an aggregation equation with degenerate diffusion term which reads as follows. Find $\rho: \bar\Omega\times[0,T]\to [0,\infty)$ such that
\begin{equation}\label{PDE}
\partial_t \rho-\Delta A(\rho)+\nabla\cdot(\rho\nabla K*\rho)=0\quad\mbox{ in }\quad Q:=\Omega\times (0,T],
\end{equation}
subject to the boundary condition
\begin{equation}\label{BC}
(-\nabla A(\rho) + \rho\nabla K * \rho) \cdot \n = 0\quad\mbox{ on }\quad \Sigma:=\partial\Omega\times(0,T]
\end{equation}
and the initial condition
\begin{equation}\label{IC}
\rho(0)=\rho^0\quad\mbox{ in }\quad \Omega,
\end{equation}
where $*$ stands for the convolution operator and $\n$ is the outward-pointing unit vector to $\partial\Omega$.

Equation \eqref{PDE} arises in many models in biology, where $\rho$ represents the population density,  $K * \rho$ stands for the density of the chemo-attractant, and $A(\rho)$ models the local repulsion. Patlak--Keller--Segel models \cite{Patlak_1953, Keller_Segel_1971, Horstmann_1970_I, Horstmann_1970_II, Hillen_Painter_2009, Boi_Cappaso_Morale_2000,Milewski_Yang_2008, Gurtin_McCamy_1977} governing the movement of species by chemotaxis are a particular instance, which correspond to considering $A(\rho)=\rho^m$ and  $K(\x)=-\frac{1}{2\pi} \log |\x|$ for $d=2$ or $K(\x)=\frac{1}{3 \omega_d d(2-d)}\frac{1}{|\x|^{d-2}}$ for $d\ge 3$, with $\omega_d$ being the volume of the unit ball in $\R^d$.  Pure aggregation equations modeling biological swarming \cite{Mogilner_Edelstein-Keshet_Bent_Spiros_2003, Mogilner_Edelstein-Keshet_1999, Topaz_Bertozzi_Lewis_2004, Topaz_Bertozzi_Lewis_2006} result from ruling  out the diffusion term $-\Delta A(\rho)$ and from selecting $K(\cdot)$ to be the Newtonian potential, repulsive-attractive Morse potential, or power law potential.

%and molecular self-assembly\cite{Doye_Wales_Berry_1995, Rechtsman_Stillinger_2006, Wales_2010}

While there is a rich body of literature on the mathematical analysis of equation \eqref{PDE} supported by numerical simulations, very few results on numerical analysis are available for the situation considered here. Carrillo, Chertock, and Huang \cite{Carrillo_Chertock_Huang_2015} introduced a positivity-preserving entropy-decreasing finite volume scheme for \eqref{PDE} which takes into account a confinement potential term as well.
%James and Vauchelet \cite{James_Vauchelet_2015} developed a finite difference method for a generalization of the one dimensional aggregation equation and proved its convergence to weak measure solutions. Finally, Betozzi and Craig \cite{Craig_Katy_2016} proposed a new particle method for a multidimensional nonlocal aggregation equation with different kernels, including Newtonian potential, repulsive-attractive Morse and power law potentials. The particle method uses kernels regularized by convolution with smooth and rapidly decreasing blob or mollifier functions. The resulting numerical scheme is mass conserving, preserves the energy decreasing property energy functional and preserves the Wasserstein gradient flow structure. Theoretical convergence rates are derived for the regularized blob method for aggregation equations. Several numerical examples are presented which confirm the theoretically predicted rates of convergence and illustrate properties of the method, like long-time existence of particle trajectories and dampening of repulsive forces from repulsive-attractive power law kernels.

The existence and uniqueness of a weak solution to equation \eqref{PDE} was established by Bertozzi and Slep\v{c}ev \cite{Bertozzi_Slepcev_2010} for $A(\rho)$ being degenerate and $K(\cdot)$ satisfying  some regularity assumptions.  It is this degeneracy of $A$ that is the major source of difficulties in studying equation \eqref{PDE}. The existence proof consists of three steps: (a) introducing a regularized problem via the diffusion term $A(\rho)$, (b) establishing a maximum principle and a priori energy bounds independent of the regularizing parameter, and (c) proving compactness for the regularized problem. In particular, the compactness of the regularized solutions is obtained by using some results borrowed from \cite{Alt_Luckhaus_1983} based on the Riesz--Fréchet--Kolmogorov criterion on Lebesgue spaces.

Our aim in this work is to construct a sequence of fully discrete approximations and analyze its convergence toward the unique solution to \eqref{PDE}-\eqref{IC}. Our algorithm uses a stabilized finite element method combined with a mass lumping technique plus a semi--implicit Euler time integration. This resulting scheme is conditionally solvable and mass conserving, and preserves nonnegativity under acute partitions of the computational domain. A priori energy bounds are obtained in a different way from those in \cite{Bertozzi_Slepcev_2010} since a discrete maximum principle does not hold. The lack of such a discrete maximum principle is overcame with the use of a nodal truncating operator \cite{GG_GS_2009}. A version of the Riesz--Fréchet--Kolmogorov compactness criterion on Lebesgue spaces by perturbation \cite{Arzela_GG} allows the passage to the limit in the nonlinear terms as the spatial and temporal discretization parameters tend to zero in order to reach the unique weak solution of \eqref{PDE}-\eqref{IC}.

%
%including Newtonian potential, repulsive-attractive Morse and power law potentials
%
%There is an extensive literature on the mathematical analysis of system \eqref{PDE}.  The key mathematical difficulty in proving existence of a weak solution to this nonlinear degenerate equation is that the diffusion coefficient of the ui equation degenerates
%
%Hablar del principio del maximo.

%Let us briefly refer to two other degenerate parabolic systems that have recently been studied.
%
%
%
%

%A numerical solution for solving an aggregation equation with degenerate diffusion is considered. The equation at hand reads as follows.

\subsection{Notation}
For $p\in[1,\infty]$, we denote by $L^p(\Omega)$ the usual Lebesgue space, i.e.,
$$
L^p(\Omega) = \{v : \Omega \to \R\, :\, v \mbox{ Lebesgue-measurable}, \int_\Omega |v(\x)|^p d\x<\infty \}.
$$
or
$$
L^\infty(\Omega) = \{v : \Omega \to \R\, :\, v \mbox{ Lebesgue-measurable}, {\rm ess}\sup_{\x\in \Omega} |v(\x)|<\infty \}.
$$

This space is a Banach space endowed with the norm
$\|v\|_{L^p(\Omega)}=(\int_{\Omega}|v(\x)|^p\,{\rm d}\x)^{1/p}$ if $p\in[1, \infty)$ or $\|v\|_{L^\infty(\Omega)}={\rm ess}\sup_{\x\in \Omega}|v(\x)|$ if $p=\infty$. In particular,  $L^2(\Omega)$ is a Hilbert space.  We shall use
$\left(u,v\right)=\int_{\Omega}u(\x)v(\x){\rm d}\x$ for its inner product and $\|\cdot\|$ for its norm.

Let $\alpha = (\alpha_1, \alpha_2, . . . , \alpha_d)\in \mathds{N}^d$ be a
multi-index with $|\alpha|=\alpha_1+\alpha_2+...+\alpha_d$, and let
$\partial^\alpha$ be the differential operator such that
$$\partial^\alpha=
\Big(\frac{\partial}{\partial{x_1}}\Big)^{\alpha_1}...\Big(\frac{\partial}{\partial{x_d}}\Big)^{\alpha_d}.$$

For $m\ge 0$ and $p\in[1, \infty)$, we define $W^{m,p}(\Omega)$ to be the Sobolev space of all functions whose $m$ derivatives are in $L^p(\Omega)$, i.e.,
$$
W^{m,p}(\Omega) = \{v \in L^p(\Omega)\,:\, \partial^k v \in L^2(\Omega)\ \forall ~ |k|\le m\}
$$ associated to the norm
\begin{align*}
\|f\|_{W^{m,p}(\Omega)}&=\left(\sum_{|\alpha|\le m} \|\partial^\alpha f\|^p_{L^p(\Omega)}\right)^{1/p} \quad  &&\hbox{for} \ 1 \leq p < \infty, \\
\|f\|_{W^{m,p}(\Omega)}&=\max_{|\alpha|\le m} \|\partial^\alpha f\|_{L^\infty(\Omega)}, \quad  && \hbox{for} \  p = \infty.
\end{align*}
For $p=2$, we denote $W^{m,2}(\Omega)=H^m(\Omega)$ and its dual as $(H^m(\Omega))'$. The dual pairing between $H^1(\Omega)$ and $(H^1(\Omega))'$ is denoted by $<\cdot, \cdot>$.

Let $X$ be a Banach space. Thus, $L^p(0,T;X)$ denotes the space of Bochner-measurable, $X$-valued functions on $(0, T)$ such that $\int_0^T\|f(s)\|^p_{X} {\rm d} s<\infty$ for $p\in [1,\infty)$ or  ${\rm ess}\sup_{s\in(0,T)}\|f(s)\|_X<\infty$ for
$p=\infty$.

\subsection{Outline of the paper } The layout of the paper is as follows. In section 2 we introduce the hypotheses for constructing the finite element approximation of \eqref{PDE} as  well as some auxiliary results. In section 3 we present our finite element method which includes a stabilizing term and combines a semi-implicit time integration. Afterwards we state our main theorem which is proved in the subsequent sections. The well-posedness of our algorithm is carried out in section 4. Non-negativity under the acuteness of the mesh and a priori energy estimates are obtained in section 5.  Section 6 deals with the compactness of the finite element approximations. The passage to the limit toward the unique weak solution of \eqref{PDE} is reported in section 7. To finish off, we present a numerical example in section 8.
\section{The discrete setting}
This section is mainly devoted to the numerical tools for approximating the solution to problem \eqref{PDE}-\eqref{IC}.
\subsection{Hypotheses} Herein we set out the hypotheses that will be required for the domain, the mesh, and the finite element space.
\begin{enumerate}
\item [(H1)] Let $\Omega$ be a convex, bounded domain of $\R^d$ with polygonal ($d=2$) or polyhedral ($d=3$) Lipschitz-continuous boundary.
\item[(H2)] Let $\{{\mathcal E}_{h}\}_{h>0}$  be a family of simplicial partitions of  $\overline{\Omega}$ that is acute, shape-regular, and quasi-uniform, so that $\overline \Omega=\cup_{E\in {\mathcal E}_h}E$, where $h=\max_{E\in \mathcal{E}_h} h_E$, with $h_E$ being the diameter of $E$. More precisely, we assume that
\begin{enumerate}
\item[(a)]  there exists $\alpha>0$, independent of $h$, such that
$$
\min\{  {\rm diam}\,  B_E\, : \, E\in\mathcal{E}_h \}\ge \alpha h,
$$
where $B_E$ is the largest ball contained in $E$, and
\item[(b)] there exists $\beta > 0$ such that  every angle between two edges (or faces)  of a triangle (or a tetrahedron) is bounded by $\frac{\pi}{2} - \beta$.
\end{enumerate}
Further, let ${\mathcal N}_h = \{\a_i\}_{i\in I}$ denote the set of all the nodes of ${\mathcal E}_h$.
\item [(H3)] A conforming finite element space associated with
${\mathcal E}_h$ is assumed for approximating $H^1(\Omega)$. Let  $\mathcal{P}_1(E)$ be the set of linear polynomials on  $E$; the space of continuous, piecewise polynomial
functions on ${\mathcal E}_h$  is then denoted as
$$
D_h = \left\{ \bar\rho_h \in {C}^0(\overline\Omega) \;:\;
\bar\rho_h|_E \in \mathcal{P}_1(E), \  \forall E \in \mathcal{E}_h \right\},
$$
whose shape functions are $\{\varphi_\a\}_{\a\in{\mathcal{N}_h}}$.
\end{enumerate}
%\begin{remark}
%For the construction of acute simplicial triangularizations $(d=2)$ or tetrahedralizations $(d=3)$  of $\bar\Omega$, we refer the readers to \cite{Eppstein_Sullivan_Ungor_2004, Brandts_Korotov_Krizek_Solc_2009} and references therein.
% \end{remark}

\subsection{Technical preliminaries} Under hypotheses $\rm(H1)$--$\rm(H3)$ we collect some properties that will be used in the subsequent analysis.

To start with, we state a consequence of the acuteness of the mesh needed for proving non-negativity of the finite element approximation.
\begin{proposition} Let $E\in\mathcal{E}_h$ with vertices $\{\boldsymbol{a}_0,\cdots\boldsymbol{a}_d\}$. Then there exists a constant $C_{\rm neg}>0$, depending on $\beta$, but otherwise independent of $h$ and $E$, such that
\begin{equation}\label{off-diagonal}
\int_E \nabla\varphi_{\a_i}\cdot\nabla\varphi_{\a_j}\dx  \le - C_{\rm neg} h^{d-2}
\end{equation} for all $\a_i,\a_j\in E$ with $i\not=j$, and
\begin{equation}\label{diagonal}
\int_E\nabla\varphi_{\a_i}\cdot\nabla\varphi_{\a_i}\dx\ge C_{\rm neg} h^{d-2}
\end{equation}
for all $\a_i\in E$.
\end{proposition}
\begin{proof} For every $d$-simplex $E\in \mathcal{E}_h$ and for every vertex $\boldsymbol{a}_i\in E$, we denote by $F_{\boldsymbol{a}_i}$ the opposite face  to $\boldsymbol{a}_i$ and by $\boldsymbol{n}_{\boldsymbol{a}_i}$ the exterior (to the $d$-simplex $E$) unit normal vector to the face $F_{\boldsymbol{a}_i}$. Write
$$
\nabla\varphi_{\a_i}|_{E}=-\frac{1}{h_{F_{\a_i}}} \n_{\a_i},
$$
where $h_{F_{\a_i}}$ is the distance of $\a_i$ to the hyperplane which contains $F_{\a_i}$. Then we have
$$
\nabla\varphi_{\a_i}|_{E}\cdot\nabla\varphi_{\a_j}|_{E}=\frac{1}{h_{F_{\a_i}}}\frac{1}{h_{F_{\a_j}}}\n_{\a_i}\cdot\n_{\a_j}.
$$
Note that $\n_{\a_i}\cdot\n_{\a_j} =\cos(\widehat{\n_{\a_i}\,\n_{\a_j}})=\cos(\widehat{F_{\a_i}\,F_{\a_j}}-\pi)=-\cos(\widehat{F_{\a_i}\,F_{\a_j}})$.
Integrating over $E$ gives
$$
\begin{array}{rcl}
\displaystyle
\int_E\nabla\varphi_{\a_i}\cdot\nabla\varphi_{\a_j}\dx&=&\displaystyle
-  |E|\frac{1}{h_{F_{\a_i}}}\frac{1}{h_{F_{\a_j}}} \cos(\widehat{F_{\a_i}F_{\a_j}})
\\
&\le& \displaystyle
-  |B_E|\frac{1}{h_{F_{\a_i}}}\frac{1}{h_{F_{\a_j}}} \cos(\frac{\pi}{2}-\beta)
\\
&=&\displaystyle
- \frac{\pi^{\frac{d}{2}}}{2^d\Gamma(\frac{d}{2}+1)} {({\rm diam}\,B_E)^d}\frac{1}{h_{F_{\a_i}}}\frac{1}{h_{F_{\a_j}}} \cos(\frac{\pi}{2}-\beta)
\\
&\le&\displaystyle
- \alpha^d \frac{\pi^{\frac{d}{2}}}{2^d\Gamma(\frac{d}{2}+1)} {h^d}\frac{1}{h_{F_{\a_i}}}\frac{1}{h_{F_{\a_j}}}\cos(\frac{\pi}{2}-\beta)
\\
&\le&\displaystyle
- \alpha^d \frac{\pi^{\frac{d}{2}}}{2^d\Gamma(\frac{d}{2}+1)} \cos(\frac{\pi}{2}-\beta) {h^{d-2}},
\end{array}
$$
where we have used that the fact that $|B_E|=\frac{\pi^{\frac{d}{2}}}{2^d\Gamma(\frac{d}{2}+1)} {({\rm diam}\,B_E)^d}$ with $\Gamma(\cdot)$ being  Euler's gamma function.

The same argument as in the proof of \eqref{off-diagonal} yields \eqref{diagonal}.
\end{proof}
Some inverse inequalities are provided in the following proposition.
\begin{proposition} Let $E\in\mathcal{E}_h$. There exists a constant $C_{\rm inv}>0$, independent of $h$ and $E$, such that, for all $\bar\rho_h\in  \mathcal{P}_1(E)$,
\begin{equation}\label{inv:H1-L2}
\|\bar \rho_h\|_{H^1(E)}\le \frac{C_{\rm inv}}{h} \|\bar\rho_h\|_{L^2(E)}
\end{equation}
and
\begin{equation}\label{inv:H-1-L2}
 \|\bar\rho_h\|_{L^2(E)} \le \frac{C_{\rm inv}}{h}\|\bar \rho_h\|_{(H^1(E))'}.
\end{equation}
\end{proposition}
\begin{proof} The proof of \eqref{inv:H1-L2} can be found in \cite[Lem. 4.5.3]{Brenner_Scott_2008} or \cite[Lem. 1.138]{Ern_Guermond_2004}.

To obtain \eqref{inv:H-1-L2}, we use a duality argument. Let $\pi_h$ be the $L^2(E)$ orthogonal interpolation operator from $L^2(E)$ into $\mathcal{P}_1(E)$. Then, from \eqref{inv:H1-L2}, we find
$$
\begin{array}{rcl}
\|\bar\rho_h\|_{L^2(E)}&=&\displaystyle
\sup_{0\not=\rho\in L^2(E)}\frac{(\bar\rho_h, \rho)}{\|\rho\|_{L^2(E)}}\le\sup_{0\not=\rho\in L^2(E)}\frac{(\bar\rho_h, \pi_h\rho)}{\|\pi_h\rho\|_{L^2(E)}}
\\
&\le&\displaystyle
\frac{ C_{\rm inv}}{h} \sup_{0\not=\rho\in L^2(\Omega)}\frac{(\bar\rho_h, \pi_h\rho)}{\|\pi_h\rho\|_{H^1(E)}}\le\frac{C_{\rm inv}}{h} \sup_{0\not=\rho_h\in D_h}\frac{(\bar\rho_h, \rho_h)}{\|\rho_h\|_{H^1(E)}}
\\
&\le&\displaystyle \frac{C_{\rm inv}}{h} \sup_{0\not=\rho\in H^1(\Omega)}\frac{(\bar \rho_h, \rho)}{\|\rho\|_{H^1(E)}}=\frac{C_{\rm inv}}{h} \|\bar\rho_h\|_{(H^1(E))'}.
\end{array}
$$
\end{proof}
\begin{corollary} There holds
\begin{equation}\label{inv:H1-L2-global}
\|\bar \rho_h\|_{H^1(\Omega)}\le \frac{C_{\rm inv}}{h} \|\bar\rho_h\|_{L^2(\Omega)}
\end{equation}
and
\begin{equation}\label{inv:H-1-L2-global}
 \|\bar\rho_h\|_{L^2(\Omega)} \le \frac{C_{\rm inv}}{h}\|\bar \rho_h\|_{(H^1(\Omega))'}.
\end{equation}
\end{corollary}

Let $\mathcal{I}_h$ be the nodal interpolation operator from $ C^0(\bar\Omega)$ to $D_h$ and  consider
$$
(\rho_h,\bar \rho_h)_h=\int_\Omega \mathcal{I}_{h}( \rho_h\bar \rho_h)=\sum_{\a\in\mathcal{N}_h}  \rho_h(\a)\cdot\bar \rho_h(\a) \int_\Omega \varphi_\a
$$
for all $ \rho_h,\bar \rho_h\in  D_h$, with the induced norm $\| \rho_h\|_h = \sqrt{( \rho_h, \rho_h)_h}$. It is well-known that there exists a constant $C_{\rm eq}>1$, independent of $h$, such that
\begin{equation}\label{Equivalence-L2}
\|\rho_h\|_h\le \|\rho_h\|_{L^2(\Omega)}\le C_{\rm eq} \|\rho_h\|_h.
\end{equation}
From the definition of $\mathcal{I}_h$, one can straightforwardly check the following.
\begin{proposition} Let $E\in\mathcal{E}_h$. It follows that
\begin{equation}\label{sta_Ih_Linf}
\|\mathcal{I}_h\varphi\|_{L^\infty(E)}\le \|\varphi\|_{L^\infty(E)}\quad\mbox{ for all }\quad\varphi\in L^\infty(E)
\end{equation}
and
\begin{equation}\label{sta_Ih_H1}
\|\nabla\mathcal{I}_h\varphi\|_{L^\infty(E)}\le \|\nabla\varphi\|_{L^\infty(E)}\quad\mbox{ for all }\quad\varphi\in W^{1,\infty}(E).
\end{equation}
\end{proposition}
\begin{corollary} There holds
\begin{equation}\label{sta_Ih_Linf-global}
\|\mathcal{I}_h\varphi\|_{L^\infty(\Omega)}\le \|\varphi\|_{L^\infty(\Omega)}\quad\mbox{ for all }\quad\varphi\in L^\infty(\Omega)
\end{equation}
and
\begin{equation}\label{sta_Ih_H1-global}
\|\nabla\mathcal{I}_h\varphi\|_{L^\infty(\Omega)}\le \|\nabla\varphi\|_{L^\infty(\Omega)}\quad\mbox{ for all }\quad\varphi\in W^{1,\infty}(\Omega).
\end{equation}
\end{corollary}

\begin{proposition} Let $E\in\mathcal{E}_h$. There exists a constant $C_{\rm app}>0$, independent of $h$ and $E$, such that
\begin{equation}\label{error-Linf-W1inf_W2inf}
\|\varphi-\mathcal{I}_h\varphi\|_{L^\infty(E)}+ h \|\nabla(\varphi-\mathcal{I}_h\varphi)\|_{L^\infty(E)} \le C_{\rm  app} h^2 \|\nabla^2 \varphi\|_{L^\infty(E)}\quad\mbox{ for all }\quad \varphi\in W^{2,\infty}(E).
\end{equation}
and
\begin{equation}\label{error-H1-H2}
\|\nabla(\varphi-\mathcal{I}_h\varphi)\|_{L^2(E)} \le C_{\rm  app} h \|\nabla^2 \varphi\|_{L^2(E)}\quad\mbox{ for all }\quad \varphi\in H^2(E).
\end{equation}
\end{proposition}
\begin{proof}  The proof of \eqref{error-Linf-W1inf_W2inf} and \eqref{error-H1-H2} can be found in \cite[Thm. 4.4.4]{Brenner_Scott_2008} or \cite[Thm. 1.103]{Ern_Guermond_2004}.
\end{proof}
\begin{corollary} There holds
\begin{equation}\label{error-Linf-W1inf_W2inf-global}
\|\varphi-\mathcal{I}_h\varphi\|_{L^\infty(\Omega)}+ h \|\nabla(\varphi-\mathcal{I}_h\varphi)\|_{L^\infty(\Omega)} \le C_{\rm  app} h^2 \|\nabla^2 \varphi\|_{L^\infty(\Omega)}\quad\mbox{ for all }\quad \varphi\in W^{2,\infty}(\Omega)
\end{equation}
and
\begin{equation}\label{error-H1-H2-global}
\|\nabla(\varphi-\mathcal{I}_h\varphi)\|_{L^2(\Omega)} \le C_{\rm  app} h \|\nabla^2 \varphi\|_{L^2(\Omega)}\quad\mbox{ for all }\quad \varphi\in H^2(\Omega).
\end{equation}
\end{corollary}
\begin{proposition} There exists a constant $C_{\rm com}>0$, independent of $h$ and $E$, such that
\begin{equation}\label{error-L1-H-1-Linf}
\|\rho_h \bar \rho_h-\mathcal{I}_h(\rho_h \overline \rho_h)\|_{L^1(\Omega)}\le C_{\rm com} h^{\frac{1}{2}}\, \|\rho_h\|_{(H^1(\Omega))'} \, \|\nabla \bar\rho_h\|_{L^\infty(\Omega)}
\end{equation}
and
\begin{equation}\label{error-L1-L2-H1}
\|\rho_h \bar \rho_h-\mathcal{I}_h(\rho_h \overline \rho_h)\|_{L^1(\Omega)}\le C_{\rm com} h\, \|\rho_h\|_{L^2(\Omega)} \, \|\nabla \bar\rho_h\|_{L^2(\Omega)}.
\end{equation}

\end{proposition}
\begin{proof} On each element $E\in\mathcal{E}_h$, combine \eqref{error-Linf-W1inf_W2inf}, \eqref{inv:H1-L2}, and \eqref{inv:H-1-L2} to obtain
$$
\begin{array}{rcl}
\|\rho_h \bar \rho_h-\mathcal{I}_h(\rho_h \overline \rho_h)\|_{L^1(E)}&\le& C_{\rm app} h^2 |E|  \|\nabla^2(\rho_h\bar\rho_h)\|_{L^\infty(E)}
\\
&\le& C_{\rm app} h^2 |E| \|\nabla\rho_h\|_{L^\infty(E)} \|\nabla\bar\rho_h\|_{L^\infty(E)}
\\
&\le& C_{\rm app} h^2 \| \nabla\rho_h\|_{L^1(E)} \|\nabla\bar\rho_h\|_{L^\infty(E)}
\\
&\le& C_{\rm app} h^\frac{5}{2} \| \nabla\rho_h\|_{L^2(E)} \|\nabla\bar\rho_h\|_{L^\infty(E)}
\\
&\le& C_{\rm app} C_{\rm inv}^2 h^\frac{1}{2} \|\rho_h\|_{(H^1(E))'} \|\nabla\bar\rho_h\|_{L^\infty(E)}.
\end{array}
$$
Estimate \eqref{error-L1-H-1-Linf} follows by summing up this last estimate over all the elements $E \in \mathcal{E}_h$.

One can prove estimate \eqref{error-L1-L2-H1} in a similar fashion.
\end{proof}

\begin{proposition} Let $f\in C^{0,1}(\mathds{R})$  be monotonically increasing with Lipschitz constant $C_{\rm Lip}$. Then  it follows that, for all $\rho_h\in D_h$,
\begin{equation}\label{Coercitivity}
|\nabla\mathcal{I}_h f(\rho_h)|^2\le C_{\rm Lip}\nabla\rho_h\cdot \nabla\mathcal{I}_h f(\rho_h).
\end{equation}
\end{proposition}
\begin{proof} On each element $E\in\mathcal{E}_h$, consider $ \tilde E$ to be an oriented, right element with vertices $\{\a_0^{\tilde E}, \cdots, \a_{d}^{\tilde E}\}$, where  $\a_0^{\tilde E}$ is the vertex supporting the right angle, such that $\tilde E\subset E$. Observe that
$$
\partial_{x_i} \rho_h=\frac{\rho_h(\a_i^{\tilde E})-\rho_h(\a_0^{\tilde E})}{(\a_i^{\tilde E}-\a_0^{\tilde E})_{i}},
$$
where $(\x)_{i}$ is the $i$th component of $\x$. Since
$$
(f(x)-f(y))^2\le C_{\rm Lip} (f(x)-f(y))(x-y)\quad\mbox{ for all }\quad x,y\in\R,
$$
we have
$$
\begin{array}{rcl}
C_{\rm Lip}\nabla\rho_h|_{E}\cdot \nabla\mathcal{I}_hf(\rho_h)|_{E}&=&\displaystyle
C_{\rm Lip}\sum_{i=1}^d \frac{\rho_h(\a_i^{\tilde E})-\rho_h(\a_0^{\tilde E})}{(\a_i^{\tilde E}-\a_0^{\tilde E})_{i}} \frac{f(\rho_h(\a_i^{\tilde E}))-f(\rho_h(\a_0^{\tilde E}))}{(\a_i^{\tilde E}-\a_0^{\tilde E})_{i}}
\\
&\ge&\displaystyle
\sum_{i=1}^d \frac{(f(\rho_h(\a_i^{\tilde E}))-f(\rho_h(\a_0^{\tilde E})))^2}{((\a_i^{\tilde E}-\a_0^{\tilde E})_i)^2}=|\nabla\mathcal{I}_hf(\rho_h)|_{E}|^2,
\end{array}
$$
where $\boldsymbol{e}_i$ is the i\emph{th} vector of the  Cartesian basis. We deduce \eqref{Coercitivity} upon summing over all the element $E\in\mathcal{E}_h$.
\end{proof}
For each element $E\in \mathcal{E}_h$ with vertices $\{\boldsymbol{a}_0,\cdots\boldsymbol{a}_d\}$,  we associate once and for all a vertex $\a_E\in E$. Thus we define $\mathcal{P}_h(\rho_{h})(\boldsymbol{x})= \rho_{h}(\boldsymbol{a}_E)$ for all $\boldsymbol{x}\in E$.
\begin{proposition} There exists a constant $C_{\rm int}>0$, independent of $h$, such that
\begin{equation}\label{error-Ph}
\|\rho_h-\mathcal{P}_h\rho_h\|_{L^2(\Omega)}\le C_{\rm int} h \|\nabla\rho_h\|_{L^2(\Omega)}\quad\mbox{ for all }\quad\rho_h\in D_h.
\end{equation}
\end{proposition}
\begin{proof} Let $\x\in E$ and write
$$
\rho_{h}(\boldsymbol{x})-\mathcal{P}_h(\rho_h)(\x)=\rho_{h}(\x)-\rho_h(\a_E)= \nabla \rho_h|_{E}\cdot(\x-\a_E).
$$
Squaring and integrating over $E$ gives
$$
\|\rho_h-\mathcal{P}_h(\rho_h)\|_{L^2(E)}\le C\, h \|\nabla \rho_h\|_{L^2(E)}
$$
 and  hence summing over $E\in \mathcal{E}_h$ yields the desired result.
\end{proof}
Moreover, let $\tilde\Delta_h$ be defined from $D_h$ to $D_h$ as
\begin{equation}\label{Discrete-Laplacian}
-(\tilde\Delta_h \phi_h, \bar \rho_h)_h=(\nabla \phi_h, \nabla \bar \rho_h)\quad \mbox{ for all } \bar \rho_h\in D_h,
\end{equation}
and let $\phi(h)\in H^2(\Omega)$ be such that
\begin{equation}\label{eq:phi(h)}
\left\{
\begin{array}{rclcl}
-\Delta\phi(h)&=&-\tilde\Delta_h \phi_h&\mbox{ in }&\Omega,
\\
\partial_{\n}\phi(h)&=&0&\mbox{ on }&\partial\Omega.
\end{array}
\right.
\end{equation}
The $H^2(\Omega)$-regularity of $\phi(h)$ is ensured by the convexity assumption stated in $\rm (H1)$. See \cite{Grisvard_1985} for a proof.
\begin{proposition} There exists a constant $C_{\rm Lap}>0$, independent of $h$, such that, for all $\phi_h\in D_h$,
\begin{equation}\label{Error-phi(h)}
\|\nabla(\phi(h)-\phi_h)\|_{L^2(\Omega)}\le C_{\rm Lap} h \|\tilde\Delta_h\phi_h\|_{L^2(\Omega)}.
\end{equation}
\end{proposition}
\begin{proof} Testing \eqref{eq:phi(h)} with $\bar\rho_h\in D_h$ yields
$$
(\nabla \phi(h), \nabla \bar\rho_h)=-(\tilde \Delta_h\phi_h, \bar \rho_h).
$$
Combining the above equation and \eqref{Discrete-Laplacian}, we write
$$
(\nabla (\phi(h)-\phi_h), \nabla \bar \rho_h)=(\tilde \Delta_h\phi_h, \bar \rho_h)_h-(\tilde \Delta_h\phi_h, \bar \rho_h)
$$
and hence
$$
(\nabla (\mathcal{I}_h\phi(h)-\phi_h), \nabla \bar \rho_h)=(\tilde \Delta_h\phi_h, \bar \rho_h)_h-(\tilde \Delta_h\phi_h, \bar \rho_h)+(\nabla (\mathcal{I}_h\phi(h)-\phi(h)), \nabla\bar\rho_h).
$$
We now choose $\bar \rho_h=\mathcal{I}_h\phi(h)-\phi_h$ to get
\begin{equation}\label{pro2.8-lab1}
\begin{array}{rcl}
\|\nabla(\mathcal{I}_h\phi(h)-\phi_h)\|^2_{L^2(\Omega)}&=& (\tilde \Delta_h\phi_h, \mathcal{I}_h\phi(h)-\phi_h)_h-(\tilde \Delta_h\phi_h, \mathcal{I}_h\phi(h)-\phi_h)
\\
&&+(\nabla (\mathcal{I}_h\phi(h)-\phi(h)), \nabla (\mathcal{I}_h\phi(h)-\phi_h)).
\end{array}
\end{equation}
By \eqref{error-L1-L2-H1} and \eqref{error-H1-H2-global}, we have
\begin{equation}\label{pro2.8-lab2}
|(\tilde \Delta_h\phi_h, \mathcal{I}_h\phi(h)-\phi_h)_h-(\tilde \Delta_h\phi_h, \mathcal{I}_h\phi(h)-\phi_h)|\le C_{\rm com} h \|\tilde \Delta_h\phi_h\|_{L^2(\Omega)} \|\nabla(\mathcal{I}_h\phi(h)-\phi_h)\|_{L^2(\Omega)}
\end{equation}
and
\begin{equation}\label{pro2.8-lab3}
(\nabla (\mathcal{I}_h\phi(h)-\phi(h)), \nabla (\mathcal{I}_h\phi(h)-\phi_h))\le C_{\rm app} h \|\tilde \Delta_h\phi_h\|_{L^2(\Omega)} \|\nabla (\mathcal{I}_h\phi(h)-\phi(h))\|_{L^2(\Omega)}.
\end{equation}
Consequently, estimate \eqref{Error-phi(h)} is satisfied by inserting \eqref{pro2.8-lab2} and \eqref{pro2.8-lab3} into \eqref{pro2.8-lab1}.
\end{proof}

\begin{corollary} There holds
\begin{equation}\label{inv:Lap-Grad-phi}
\|-\tilde\Delta_h\phi_h\|_{ L^2(\Omega)}\le \frac{C_{\rm inv}}{h}  \|\nabla\phi_h\|_{L^2(\Omega)}
\end{equation}
and
\begin{equation}\label{phi(h)-phi-H1}
\|\nabla\phi(h)\|_{L^2(\Omega)}\le C_{\rm sta} \|\nabla\phi_h\|_{ L^2(\Omega)}.
\end{equation}
\end{corollary}
\begin{proof} Select $\bar\rho_h=-\tilde\Delta_h\phi_h$ in \eqref{Discrete-Laplacian} and use  \eqref{inv:H1-L2} to have
$$
\|\tilde\Delta\phi_h\|_{L^2(\Omega)}^2\le \|\nabla\phi_h\|_{L^2(\Omega)} \|\nabla\tilde\Delta_h\phi_h\|_{L^2(\Omega)}\le \frac{C_{\rm inv}}{h} \|\nabla\phi_h\|_{L^2(\Omega)} \|\tilde\Delta_h\phi_h\|_{L^2(\Omega)},
$$
which implies \eqref{inv:Lap-Grad-phi}. Inequality \eqref{phi(h)-phi-H1} is obtained by using  \eqref{Error-phi(h)} and \eqref{inv:Lap-Grad-phi}, so we find that
$$
\|\nabla\phi(h)\|_{L^2(\Omega)}\le \|\nabla\phi_h\|_{L^2(\Omega)}+C_{\rm Lap}h\|\tilde\Delta_h\phi_h\|_{L^2(\Omega)}\le (1+ C_{\rm Lap} C_{\rm inv})\|\nabla\phi_h\|_{L^2(\Omega)}.
$$
\end{proof}
In order to construct a proper sequence of initial approximations we need an interpolation operator that preserves non-negativity and has $L^p$-stability.  Let $\mathcal{SZ}_h$ be the variant of the Scott-Zhang interpolation operator defined in \cite{Girault_Lions_2001}, which satisfies the following.
\begin{proposition} For $p\in[1,\infty]$, $s=0,1$, and $m=0,1$, there exist two constants $C_{\rm sta}, C_{\rm app}>0$, independent of $h$, such that
\begin{equation}\label{est:SZ-stability}
\|\mathcal{SZ}_h\varphi\|_{W^{s,p}(\Omega)}\le C_{\rm sta}\|\varphi\|_{W^{s,p}(\Omega)}\quad\mbox{  for all }\quad \varphi\in W^{s,p}(\Omega).
\end{equation}
and
\begin{equation}\label{error-SZ}
\|\varphi-\mathcal{SZ}_h\varphi\|_{W^{s,p}(\Omega)}\le C_{\rm  app} h^{m+1-s} \|\varphi\|_{W^{m+1,p}(\Omega)}\quad\mbox{ for all }\quad \varphi\in W^{m+1,p}(\Omega).
\end{equation}
Moreover,
\begin{equation}\label{positivity-SZ}
\mbox{ if }\quad \varphi\ge0\quad \mbox{ in }\quad\Omega,\quad \mbox { then }\quad\mathcal{SZ}_h\varphi\ge 0\quad\mbox{ in }\quad \Omega.
\end{equation}
\end{proposition}

Henceforth $C$ denotes a generic constant whose value may change at each occurrence. This constant may depend on the data problem and the constants $C_{\rm neg}$, $C_{\rm inv}$, $C_{\rm eq}$,  $C_{\rm app}$, $C_{\rm Lip}$,  $C_{\rm com }$, and $C_{\rm Lap }$.
\section{Statement of the main result}
Let $\rho^0\in L^\infty(\Omega)$ be nonnegative and consider $\rho^0_h=\mathcal{SZ}_h\rho^0$. From \eqref{est:SZ-stability} and \eqref{positivity-SZ},  we see that
\begin{equation}\label{Initial-Positivity}
\rho^0_h\ge0\quad \mbox{ in }\quad\Omega
\end{equation}
and
\begin{equation}\label{initial-boundness}
\| \rho_{h}^0\|_{L^p(\Omega)}\le C_{\rm sta} \|\rho^0\|_{L^p(\Omega)}\quad\mbox{ for all }\quad p\in[1,\infty].
\end{equation}
Moreover, a regularization argument together with \eqref{error-SZ} provides
\begin{equation}\label{initial-convergences}
\rho^0_h\to \rho^0\quad\mbox{ in }\quad L^p(\Omega)\mbox{-strongly as}\quad h\to0.
\end{equation}

Let $k=\frac{T}{N}$ with $N\in\mathds{N}$ and consider $\{t_n\}_{n=0}^N $ with $t_n=k\,n$.
Given $\rho^n_h\in D_h$, compute $\rho^{n+1}_h\in D_h$ satisfying
\begin{equation}\label{Scheme}
(\delta_t\rho^{n+1}_h,\bar\rho_h)_h+ h^\gamma (\nabla\rho^{n+1}_h,\nabla\bar\rho_h)+(\nabla\mathcal{I}_h  A([\rho^{n+1}_h]_T),\nabla\bar\rho_h)-(\rho^{n+1}_h\nabla \mathcal{I}_h( K*[\rho^n_h]_T),\nabla\bar\rho_h)=0,
\end{equation}
where $0<\gamma<1$  and $[\cdot]_T:D_h\to D_h$ is a nodal truncating operator defined as
$$
[\bar\rho_h(\a)]_T=\left\{
\begin{array}{rcl}
0&\mbox{ if }&\bar\rho_h(\a)\in(-\infty, 0),
\\
\bar\rho_h(\a)&\mbox{ if }& \bar \rho_h(\a)\in [0,B_{L^\infty}],
\\
B_{L^\infty}&\mbox{ if }&\bar \rho_h(\a)\in(B_{L^\infty}, +\infty),
\end{array}
\right.
$$
with $\a\in\mathcal{N}_h$ and $B_{L^\infty}= e^{ T\|\Delta K\|_{L^{\infty}(\mathds{R}^d)}\|\rho_{0}\|_{L^1(\Omega)}}\|\rho_0\|_{L^\infty(\Omega)}$. By an abuse of notation, the convolution $K*[\rho_{h,k}]_T$ must be understood for $[\rho_{h,k}]_T \chi_\Omega$, where $\chi_\Omega$ is the characteristic function. Moreover, the definition of $B_{L^\infty}$ will be explained later on.

Due to the embedding $W^{2,\infty}(\R^d)$ into $C^0(\R^d)$, the convolution $K*[\rho_{h,k}]_T$ belongs to $ C^0(\R^d)$; therefore, we are allowed to consider $\mathcal{I}_h ((K*[\rho_{h,k}]_T)|_{\Omega})$ that we write $\mathcal{I}_h (K*[\rho_{h,k}]_T)$  to simplify notation. Further, we have introduced $\delta_t\rho^{n+1}_h=\frac{\rho^{n+1}_h-\rho^n_h}{k}$.

For future references, note that $\mathcal{I}_hA([\rho_h^{n+1}]_T)=\mathcal{I}_h A_T(\rho^{n+1}_h)$, where
$$
A_T(s)=\left\{\begin{array}{ccl} 0 &\mbox{ if }& s\in(-\infty,0), \\ A(s)&\mbox{ if }& s\in[0, B_{L^\infty}],\\ A(B_{L^\infty})&\mbox{ if } & s\in(B_{L^\infty},+\infty). \end{array}\right.
$$

A weak solution for \eqref{PDE} will be understood in the following sense \cite{Bertozzi_Slepcev_2010}.
\begin{definition}\label{Weak_Sol} A function $\rho : Q\to [0, \infty)$ is a weak solution to \eqref{PDE} with \eqref{BC} and \eqref{IC} if
$$
\rho \in L^\infty(Q ),  \quad  A(\rho) \in L^2(0, T, H^1(\Omega)), \quad \partial_t\rho\in L^2(0,T; (H^1(\Omega))'),
$$
and
\begin{equation}\label{PDE-L2H-1}
\left\{
\begin{array}{rcccl}
\partial_t\rho -\Delta A(\rho)+\nabla\cdot(\rho(\nabla K*\rho))&=&0&\mbox{ in }& L^2(0,T; (H^1(\Omega))'),
\\
\rho(0)&=&\rho_0&\mbox{ in }& (H^1(\Omega))'.
\end{array}
\right.
\end{equation}
\end{definition}

To establish convergence of the discrete solutions constructed via scheme \eqref{Scheme} toward the weak solution to \eqref{PDE}, we need to assume that
\begin{enumerate}
\item[(A1)] $A$ is  $C^1([0,\infty); [0,\infty))$  with $A'>0$  on $(0, \infty)$ and $A(0) = 0$,
\end{enumerate}
and
\begin{enumerate}
\item[(K1)]  $K \in W^{2,\infty}(\R^d)$ is  such that $K(\x)=r(|\x|)$ with $r$ being nonincreasing.
%\item[(K2)] $\int_{\R^d} K(x)dx = 1$,
%\item[(K2)] $K$ is  radial, i.e., $K(\x)=k(\x)$, and $k$ is nonincreasing.
\end{enumerate}

%It is well to point out, here, that the assumptions on $K$ has been simplified concerning those in \cite{Bertozzi_Slepcev_2010} which are thought to be adjusted to real applications.

Let us define $\rho_{h,k}, \rho_{h,k}^-, \rho_{h,k}^+ :[0,T]\to D_h$ such that
$$
\rho_{h,k}=\frac{t-t_{n+1}}{k}\rho^{n+1}_h+\frac{t_n-t}{k}\rho^n_h, \quad t\in[t_n, t_{n+1}],
$$
$$
\rho_{h,k}^-=\rho_h^n, \quad \rho_{h,k}^+=\rho^{n+1}_h, \quad t\in(t_n,t_{n+1}].
$$

Our main result is summarized in the following theorem.
\begin{theorem}\label{Thm:Main} Suppose that $\rm (A1)$, $\rm(K1)$, and $\rm(H1)$-$\rm(H3)$ are satisfied. Then
\begin{enumerate}
\item there is a unique solution, $\rho_h^{n+1}$, to scheme \eqref{Scheme} provided that
\begin{equation}\label{Restriction_II}
Ck (1 + \frac{1}{h}) \|K\|_{W^{2,\infty}(\R^d)} \|\rho_h^n\|_{L^1(\Omega)}\le\frac{1}{2}.
\end{equation}
\item $\rho^{n+1}_h\ge 0$
provided that
\begin{equation}\label{Restriction_III}
C h^{1-\gamma} \|K\|_{W^{2,\infty}(\R^d)} \|\rho_h^n\|_{L^1(\Omega)}< C_{\rm neg}.
\end{equation}
\item and
$$
\int_\Omega \rho^{n+1}_h(\x)\dx=\int_\Omega \rho^{n}_h(\x)\dx.
$$
\end{enumerate}

Thus, the sequences of approximate solutions $\{\rho_{h,k}\}_{h,k>0}$ and $\{\rho_{h,k}^\pm\}_{h,k>0}$ constructed via scheme~\eqref{Scheme}
\begin{enumerate}
\item are well-defined if
\begin{equation}\label{Restriction_II-global}
Ck (1 + \frac{1}{h}) \|K\|_{W^{2,\infty}(\R^d)} \|\rho_h^0\|_{L^1(\Omega)}\le\frac{1}{2}
\end{equation}
\item satisfy
$$
\rho_{h,k}, \rho_{h,k}^\pm\ge0 \quad\mbox{ in }\quad  Q,
$$
if
\begin{equation}\label{Restriction_III-global}
C h^{1-\gamma} \|\nabla K\|_{W^{2,\infty}(\R^d)} \|\rho_h^0\|_{L^1(\Omega)}< C_{\rm neg}.
\end{equation}
\item and
$$
\int_\Omega \rho_{h,k}(t,\x)\dx=\int_\Omega \rho_{h,k}^\pm(t,\x)\dx=\int_\Omega \rho_{0h}(\x)\dx\quad\mbox{ for all }\quad t\in[0,T].
$$
\end{enumerate}
Furthermore, the sequences of approximate solutions $\{\rho_{h,k}\}_{h,k>0}$ and $\{\rho_{h,k}^\pm\}_{h,k>0}$ converge toward the unique weak solution $\rho$ of \eqref{PDE}, as $(h,k)\to (0,0)$, in the sense that
$$
[\rho_{h,k}]_T, [\rho_{h,k}^\pm]_T\to \rho \mbox{ in } L^2(0,T; L^p(\Omega))\mbox{-strongly }
$$
and
$$
A([\rho_{h,k}^+]_T)\to A(\rho) \mbox{ in } L^2(0,T; L^p(\Omega))\mbox{-strongly and in } L^2(0,T; H^1(\Omega))\mbox{-weakly}
$$
with $1<p<\infty$.
\end{theorem}

From now on we assume that assumptions $\rm (A1)$, $\rm(K1)$, and $\rm(H1)$-$\rm(H3)$ hold  without further comment on the statement of the results.

\section{Existence and uniqueness of discrete solutions }
In this section we prove the unique solvability of  scheme \eqref{Scheme}. To simplify notation we suppress the superscript in $\rho^{n+1}_h$ since there will be no ambiguity in setting $\rho_h=\rho_h^{n+1}$.

Before proceeding, we need an auxiliary result concerning the sign of $\nabla(K*[\rho^n_h]_T)\cdot \boldsymbol{n}$ on $\partial\Omega$.
\begin{lemma} It follows that
\begin{equation}\label{Kxn}
\nabla (K*[\rho^n_h]_T)\cdot \n\le 0\quad\mbox{ on }\quad\partial\Omega.
\end{equation}
\end{lemma}
\begin{proof} Let $\x\in\partial\Omega$ be such that $\n(\x)$ is well-defined at $\x\in\partial\Omega$ as being the outward unit normal vector and let $s>0$.  Write
$$
\begin{array}{rcl}
K*[\rho_h^n]_T(\x+  s \n)-K*[\rho_h^n]_T(\x)&=&
\displaystyle
\int_\Omega (K(\x+  s \n-\y)-K (\x-\y))[\rho^n_h]_T(\y)d\y
\\[2.0ex]
&=&
\displaystyle
\int_\Omega (r(|\x+  s \n-\y|)-r(|\x-\y|))[\rho^n_h]_T(\y)d\y.
\end{array}
$$
In virtue of the decreasing property from $\rm(K1)$ and the convexity from $\rm (H1)$, we find that $r(|\x+  s \n-\y|)-r(|\x-\y|)\le 0$ since $|\x+  s \n-\y|\ge |\x-\y|$ and $[\rho^n_h]_T\ge 0$. Then
$$
\partial_\n (K*[\rho^n_h]_T)(\x)=\lim_{s\to 0^+}\frac{(K*[\rho^n_h]_T)(\x+  s \n)-(K*[\rho^n_h]_T)(\x)}{s}\le 0.
$$
\end{proof}

The next lemma shows that scheme \eqref{Scheme} has at least one solution.  In doing so, we make use of Brouwer's theorem.

\begin{lemma}[Existence] Let $\rho^n_h\in L^1(\Omega)$ be such that $\rho_h^n\ge0$ in $\Omega$.  Then scheme \eqref{Scheme} has at least one solution provided that
\begin{equation}\label{Restriction_I}
C k \|K\|_{W^{2,\infty}(\R^d)} \|\rho_h^n\|_{L^1(\Omega)}\le\frac{1}{2}.
\end{equation}
for a certain constant $C>0$ independent of $h$.
\end{lemma}
\begin{proof} %First observe that $\|\nabla K_h*[\rho^n_h]_T\|_{L^\infty(\Omega)}\le \|\nabla K_h\|_{L^\infty(\Omega)} \|\rho_h^n\|_{L^1(\Omega)}$.
Let $\Phi: D_h\to D_h$ be defined  by  $\Phi(\tilde \rho_h)=\rho_h$ such that
\begin{equation}\label{eq:Phi}
\begin{array}{rcl}
(\rho_h-\rho_h^n,\bar\rho_h)_h+k h^\gamma(\nabla \rho_h,\nabla\bar\rho_h)+k(\nabla\mathcal{I}_h  A_T( \tilde\rho_h),\nabla\bar\rho_h)-k(\rho_h\nabla \mathcal{I}_h(K*[\rho^n_h]_T),\nabla\bar\rho_h)=0.
\end{array}
\end{equation}
Pick $\bar\rho_h=\rho_h$ to get
\begin{equation}\label{lm4.2-lab1}
\begin{array}{rcl}
\|\rho_h\|^2_h+k  h^\gamma\|\nabla \rho_h\|^2_{L^2(\Omega)}= (\rho_h^n,\rho_h)_h -k (\nabla\mathcal{I}_h A_T(\tilde\rho_h), \nabla \rho_h)+ k (\rho_h \nabla\mathcal{I}_h (K*[\rho_h^n]_T), \nabla \rho_h).
\end{array}
\end{equation}
Cauchy-Schwarz' and Young's inequalities give
\begin{equation}\label{lm4.1-lab1-bis}
(\rho_h^n,\rho_h)_h\le 2 \|\rho_h^n\|^2_h+\frac{1}{4}\|\rho_h\|^2_h.
\end{equation}
The second term on the right-hand of \eqref{lm4.2-lab1} can be estimated on noting \eqref{inv:H1-L2-global} as
\begin{equation}\label{lm4.2-lab1-bis}
\begin{array}{rcl}
k (\nabla\mathcal{I}_h  A_T(\tilde\rho_h), \nabla \rho_h)&\le&
\displaystyle
k\|\nabla\mathcal{I}_h A_T(\tilde\rho_h)\|_{L^2(\Omega)} \|\nabla \rho_h\|_{L^2(\Omega)}
\\
&\le&\displaystyle
\frac{1}{2} \frac{k}{h^{1+\gamma}} \|\mathcal{I}_h  A_T(\tilde\rho_h)\|^2_{L^2(\Omega)}+ \frac{1}{2}k h^\gamma \|\nabla \rho_h\|^2_{L^2(\Omega)}
\\
&\le&
\displaystyle
\frac{1}{2} |\Omega| \frac{k}{h^{1+\gamma}}  A(B_{L^\infty})+ \frac{1}{2} k  h^\gamma \|\nabla \rho_h\|^2_{L^2(\Omega)}.
\end{array}
\end{equation}
The third term on the right-hand side of \eqref{lm4.2-lab1} can be rewritten as
\begin{equation}\label{lm4.2-lab1-tris}
k (\rho_h \nabla \mathcal{I}_h (K*[\rho_h^n]_T), \nabla \rho_h)= k (\rho_h \nabla K*[\rho_h^n]_T, \nabla \rho_h) +k (\rho_h \nabla(\mathcal{I}_h-\mathcal{I})(K*[\rho_h^n]_T), \nabla \rho_h),
\end{equation}
where $\mathcal{I}$ is the identity operator. Integrating by parts and using \eqref{Kxn} leads to
\begin{equation}\label{lm4.2-lab2}
\begin{array}{rcl}
k (\rho_h \nabla K*[\rho_h^n]_T, \nabla \rho_h)&= &\displaystyle - \frac{k}{2}(\Delta K*[\rho_h^n]_T,  \rho_h^2)+ k((\nabla K*[\rho^n_h]_T)\cdot\n,\rho_h^2)_{\partial\Omega}
\\
&\le &C k \|\Delta K\|_{L^\infty(\R^d)} \|\rho_h^n\|_{L^1(\Omega)}\|\rho_h\|^2_{L^2(\Omega)},
\end{array}
\end{equation}
where we have used the fact that $\|[\rho_h^n]_T\|_{L^1(\Omega)}\le\|\rho_h^n\|_{L^1(\Omega)}$ since $\rho_h^n\ge0$. Now note that \eqref{error-Linf-W1inf_W2inf-global} and \eqref{inv:H1-L2} imply that
\begin{equation}\label{lm4.2-lab3}
\begin{array}{rcl}
k (\rho_h \nabla(\mathcal{I}_h-\mathcal{I})(K*[\rho_h^n]_T), \nabla \rho_h)&\le& C k h 
\|\nabla^2 K\|_{L^{\infty}(\R^d)} \|[\rho_h^n]_T\|_{L^1(\Omega)} \|\rho_h\|_{L^2(\Omega)} \|\nabla\rho_h\|_{L^2(\Omega)}
\\
&\le& C k\|\nabla^2 K\|_{L^{\infty}(\R^d)} \|\rho_h^n\|_{L^1(\Omega)} \|\rho_h\|^2_{L^2(\Omega)}.
\end{array}
\end{equation}
Combining \eqref{lm4.2-lab2} and \eqref{lm4.2-lab3}, we find
\begin{equation}\label{lm4.2-lab4}
k (\rho_h \nabla \mathcal{I}_h(K*[\rho_h^n]_T), \nabla \rho_h)\le C k \|K\|_{W^{2,\infty}(\R^d)} \|\rho_h^n\|_{L^1(\Omega)} \|\rho_h\|^2_{L^2(\Omega)}.
\end{equation}
Putting \eqref{lm4.2-lab1}, \eqref{lm4.1-lab1-bis}, \eqref{lm4.2-lab1-bis}, and \eqref{lm4.2-lab4} together, we arrive at the estimate
$$
\frac{3}{2}\|\rho_h\|^2_h+k  h^\gamma\|\nabla \rho_h\|^2_{L^2(\Omega)}\le4\|\rho^n_h\|^2_h+ |\Omega| \frac{k}{h^{1+\gamma}}  A(B_{L^\infty})+ C k \|K\|_{W^{2,\infty}(\R^d)} \|\rho_h^n\|_{L^1(\Omega)} \|\rho_h\|^2_{L^2(\Omega)}
$$
and, in view of \eqref{Restriction_I} and \eqref{Equivalence-L2},
$$
\|\rho_h\|_*=:\|\rho_h\|^2_h+k h^\gamma\|\nabla \rho_h\|^2_{L^2(\Omega)}\le 4\|\rho^n_h\|^2_h+ \frac{k}{h^{1+\gamma}}|\Omega| A(B_{L^\infty}):=R.
$$
As a result, if we choose $r>R$, we find that $\|\tilde\rho_h\|_*<r$ implies that $\|\Phi(\tilde\rho_h)\|_*<r$.

Let us see that $\Phi$ is a continuous mapping from $D_h$ into $D_h$ with respect to the $\|\cdot\|_*$-norm. Suppose that $\tilde\rho_{h,m}\to \tilde\rho_h$ in the $\|\cdot\|_*$-norm as $m\to\infty$. Then we want to prove that $\Phi(\tilde\rho_{h,m})\to \Phi(\tilde\rho_h)$ in the $\|\cdot\|_*$-norm as $m\to 0$. To do this, we compare \eqref{eq:Phi} and \eqref{eq:Phi} for $\rho_h=\tilde\rho_{h,m}$, and test against $\rho_h=\tilde\rho_{h,m}-\rho_h$ to get
$$
\begin{array}{rcl}
\|\Phi(\tilde\rho_{h,m})-\Phi(\tilde\rho_{h})\|^2_*&=&-k(\nabla\mathcal{I}_h A_T(\tilde \rho_{h,m})-\nabla\mathcal{I}_h A_T(\tilde\rho_h),\nabla(\Phi(\tilde\rho_{h,m})-\Phi(\tilde\rho_{h})))
\\
&&+k ((\Phi(\tilde\rho_{h,m})-\Phi(\tilde\rho_{h})) \nabla\mathcal{I}_h (K*[\rho_h^n]_T), \nabla(\Phi(\tilde\rho_{h,m})-\Phi(\tilde\rho_{h}))).
\end{array}
$$

It is straightforward to see that
\begin{align*}
-k(\nabla\mathcal{I}_h A_T(\tilde \rho_{h,m})&-\nabla\mathcal{I}_h A_T(\tilde\rho_h),\nabla(\Phi(\tilde\rho_{h,m})-\Phi(\tilde\rho_{h})))
\\
&\le k\|\nabla\mathcal{I}_h A_T(\tilde \rho_{h,m})-\nabla\mathcal{I}_h A_T(\tilde\rho_h)\|_{L^2(\Omega)} \| \nabla(\Phi(\tilde\rho_{h,m})-\Phi(\tilde\rho_{h}))\|_{L^2(\Omega)}
\\
&\le \frac{1}{2}\frac{k}{h^\gamma}\|\nabla\mathcal{I}_h A_T(\tilde \rho_{h,m})-\nabla\mathcal{I}_h A_T(\tilde\rho_h)\|^2_{L^2(\Omega)}+\frac{1}{2} k  h^\gamma\|\nabla(\Phi(\tilde\rho_{h,m})-\Phi(\tilde\rho_{h}))\|^2_{L^2(\Omega)}.
\end{align*}
and, from \eqref{lm4.2-lab4},
\begin{align*}
k ((\Phi(\tilde\rho_{h,m})&-\Phi(\tilde\rho_{h})) \nabla\mathcal{I}_h (K*[\rho_h^n]_T), \nabla(\Phi(\tilde\rho_{h,m})-\Phi(\tilde\rho_{h})))
\\
\le &C k \| K\|_{W^{2,\infty}(\R^d)} \|\rho_h^n\|_{L^1(\Omega)} \|\Phi(\tilde\rho_{h,m})-\Phi(\tilde\rho_{h})\|^2_{L^2(\Omega)}.
\end{align*}
Therefore, under \eqref{Restriction_I}, we finally get
$$
\|\Phi(\tilde\rho_{h,m})-\Phi(\tilde\rho_{h})\|^2_*\le \frac{k}{h^\gamma}\|\nabla\mathcal{I}_h A_T(\tilde \rho_{h,m})-\nabla\mathcal{I}_h A_T(\tilde\rho_h)\|^2_{L^2(\Omega)}.
$$
As we are dealing with a finite-dimensional space, all norms are equivalent in $D_h$; and therefore we infer that $\tilde\rho_{h,m}\to \tilde\rho_h$ in  $C(\bar\Omega)$ as $m\to \infty$. Since $ A_T$ is a continuous operator, we obtain that $ A_T(\tilde\rho_{h,m})\to  A_T(\tilde\rho_h)$ in $C(\bar\Omega)$ as $m\to\infty$. This gives that  $\mathcal{I}_h A_T(\tilde\rho_{h,m})\to \mathcal{I}_h A_T(\tilde\rho_h)$ in $H^1(\Omega)$ as $m\to\infty$. Now the continuity of $\Phi$ is obvious.

Next apply the Brouwer fixed-point theorem to conclude the proof.
\end{proof}

Once we have proved existence, we turn to the question of uniqueness.

\begin{lemma}[Uniqueness] Let $\rho^n_h\in L^1(\Omega)$ such that $\rho_h^n\ge0$ in $\Omega$. Then  scheme \eqref{Scheme} possesses at most one solution provided that
\eqref{Restriction_II} holds.
\end{lemma}
\begin{proof} Suppose that there are two solutions $\rho^1_h$ and $\rho_h^2$, respectively. Define $\rho_h=\rho_h^1-\rho_h^2$ which satisfies
\begin{equation}\label{lm4.3-lab1}
\frac{1}{k}(\rho_h,\bar \rho_h)_h+ h^\gamma (\nabla\rho_h,\nabla\bar\rho_h)+(\nabla\mathcal{I}_h(A([\rho_h^1]_T)-A([\rho_h^2]_T)), \nabla\bar\rho_h)-(\rho_h \nabla \mathcal{I}_h(K*[\rho^n_h]_T),\nabla\bar\rho_h)=0.
\end{equation}
Let us define $\phi_h\in D_h$ such that
\begin{equation}\label{lm4.3-lab2}
(\nabla\phi_h,\nabla\bar\rho_h)=(\rho_h,\bar\rho_h)_h\quad\mbox{ for all }\quad \bar\rho_h\in D_h.
\end{equation}
Select $\bar\rho_h=\phi_h$ in \eqref{lm4.3-lab1}  to get
\begin{equation}\label{lm4.3-lab3}
\|\nabla\phi_h\|^2_{L^2(\Omega)}+ k\,h^\gamma\|\rho_h\|_h^2=-k(\nabla\mathcal{I}_h(A([\rho_h^1]_T)-A([\rho_h^2]_T)), \nabla\phi_h)+k(\rho_h \nabla\mathcal{I}_h (K*[\rho_h^n]_T), \nabla\phi_h).
\end{equation}
The first term on the right-hand side has negative sign. Indeed, by \eqref{lm4.3-lab2}, the mean-value theorem and $\rm (A1)$,
\begin{align}
\nonumber
-k(\nabla\mathcal{I}_h(A([\rho_h^1]_T)-A([\rho_h^2]_T)), \nabla\phi_h)&=-k(\mathcal{I}_h (A([\rho_h^1]_T)-A([\rho_h^2]_T)), \rho_h )_h
\\
&=
\label{lm4.3-lab4}
\displaystyle
-k\sum_{\a\in \mathcal{N}_h} A'(\xi_\a)([\rho_h^1(\a)]_T-[\rho_h^2(\a)]_T) \rho_h(\a)  \int_\Omega\varphi_a(\x)\, \dx
\\
\nonumber
&\le
\displaystyle
-k\sum_{\a\in \mathcal{N}_h} A'(\xi_\a) ([\rho_h^1(\a)]_T-[\rho_h^2(\a)]_T)^2  \int_\Omega\varphi_a(\x)\, \dx\le 0,
\end{align}
where $\xi_\a\in ([\rho_h^1(\a)]_T,[\rho_h^2(\a)]_T)$ or $ ([\rho_h^2(\a)]_T,[\rho_h^1(\a)]_T)$.
For the second term, we proceed as follows. Combing  \eqref{Discrete-Laplacian}  and \eqref{lm4.3-lab2}, we have $-\tilde\Delta_h\phi_h=\rho_h$. Thus, by
\eqref{eq:phi(h)}, we write
$$
\begin{array}{rcl}
k(\rho_h\nabla\mathcal{I}_h( K*[\rho^n_h]_T), \nabla\phi_h)&=&-k(\Delta\phi(h)\nabla K*[\rho^n_h]_T, \nabla\phi(h))
\\
&&-k(\tilde\Delta\phi_h\nabla (\mathcal{I}_h-\mathcal{I})(K*[\rho^n_h]_T), \nabla\phi(h))
\\
&&-k(\tilde\Delta_h\phi_h\nabla\mathcal{I}_h (K*[\rho^n_h]_T), \nabla(\phi_h-\phi(h))).
\end{array}
$$
Integration by parts shows that
$$
\begin{array}{rcl}
-k(\Delta\phi(h)\nabla K*[\rho^n_h]_T, \nabla\phi(h))&=&k((\nabla\phi(h)\cdot\nabla)\nabla K*[\rho_h^n]_T, \nabla\phi(h))
\\
&&+k((\nabla K*[\rho_h^n]_T\cdot\nabla)\nabla\phi(h),\nabla\phi(h))
\\
&=&\displaystyle
k((\nabla\phi(h)\cdot\nabla)\nabla K*[\rho_h^n]_T, \nabla\phi(h))
\\
&&\displaystyle
- \frac{k}{2}(\Delta K*[\rho_h^n]_T, |\nabla\phi(h)|^2)
\\
&&\displaystyle
+\frac{k}{2}(|\nabla \phi(h)|^2, (\nabla K*[\rho_h^n]_T)\cdot\n)_{\partial\Omega},
\end{array}
$$
which, from \eqref{phi(h)-phi-H1} and \eqref{Kxn}, gives
$$
-k(\Delta\phi(h)\nabla \mathcal{I}_h(K*[\rho^n_h]_T), \nabla\phi(h))\le C k \|\Delta K\|_{L^{\infty}(\R^d)}\| \rho_h^n\|_{L^1(\Omega)} \|\nabla\phi_h\|^2_{L^2(\Omega)}.
$$
In view of \eqref{error-Linf-W1inf_W2inf-global}, \eqref{eq:phi(h)} \eqref{inv:Lap-Grad-phi}, we have
$$
-k(\Delta\phi(h)\nabla (\mathcal{I}_h-\mathcal{I})(K*[\rho^n_h]_T), \nabla\phi(h))\le C k \|\nabla^2 K\|_{L^{\infty}(\R^d)} \|\rho_h^n\|_{L^1(\Omega)} \|\nabla \phi_h\|^2_{L^2(\Omega)}.
$$
Using \eqref{Error-phi(h)}, \eqref{inv:Lap-Grad-phi}, and \eqref{sta_Ih_H1},  leads to the estimate
$$
\begin{array}{rcl}
-k(\tilde\Delta_h\phi_h\nabla \mathcal{I}_h( K*[\rho^n_h]_T), \nabla(\phi_h-\phi(h)))&\le&k\, h \|\tilde\Delta_h\phi_h\|^2_{L^2(\Omega)}\|\nabla K\|_{L^\infty(\R^d)}\|\rho^n_h\|_{L^1(\Omega)}
\\
&\le&\displaystyle
C\frac{k}{h} \|\nabla K\|_{L^{\infty}(\R^d)} \|\rho_h^n\|_{L^1(\Omega)}\|\nabla\phi_h\|^2_{L^2(\Omega)}.
\end{array}
$$
Therefore,
\begin{equation}\label{lm4.3-lab5}
k(\rho_h\nabla\mathcal{I}_h (K*[\rho^n_h]_T), \nabla\phi_h)\le Ck (1 + \frac{1}{h}) \| K\|_{W^{2,\infty}(\R^d)} \|\rho_h^n\|_{L^1(\Omega)}\|\nabla\phi_h\|^2_{L^2(\Omega)}.
\end{equation}

From \eqref{lm4.3-lab4} and \eqref{lm4.3-lab5}, we estimate \eqref{lm4.3-lab3} as
$$
\|\nabla\phi_h\|^2_{L^2(\Omega)}+k\,h^\gamma\|\rho_h\|_h^2\le Ck (1 + \frac{1}{h}) \|K\|_{W^{2,\infty}(\R^d)} \|\rho_h^n\|_{L^1(\Omega)}\|\nabla\phi_h\|_{L^2(\Omega)}^2.
$$
The result follows by taking $k/h$ small enough so that \eqref{Restriction_II} holds.
\end{proof}

It should be noted that condition \eqref{Restriction_II} is indeed more demanding than condition \eqref{Restriction_I} concerning the space and time parameters.

\section{Non-negativity and a priori estimates}
In this section we show that the discrete solution $\rho^{n+1}_h$ computed by \eqref{Scheme}  is nonnegative. Moreover, we derive some a priori energy estimates.

\begin{lemma}[Non-negativity]\label{lm:Positivity} Let $\rho^n_h\in L^1(\Omega)$ be such that $\rho_h^n\ge0$ in $\Omega$. Assume that \eqref{Restriction_II} is satisfied. Then the solution $\rho^{n+1}_h$ to scheme \eqref{Scheme} is nonnegative provided that \eqref{Restriction_III} holds.
\end{lemma}
\begin{proof} First of all, note that, for all $E\in\mathcal{E}_h$ and for all $\a_i,\a_j\in E$ with $i\not=j$,
\begin{equation}\label{lm5.1-lab1}
\begin{array}{rcl}
\displaystyle
\int_E\varphi_{\a_i} \nabla \mathcal{I}_h (K*[\rho_h^n]_T)\cdot\nabla\varphi_{\a_j}\dx&\le& |E| \|\varphi_{\a_j}\|_{L^\infty(E)} \|\nabla\mathcal{I}_h(K*[\rho_h^n]_T)\|_{L^\infty(E)} \|\nabla\varphi_{\a_j}\|_{L^\infty(E)}
\\
&\le& C h^{d-1}  \|\nabla K\|_{L^{\infty}(\R^d)} \|\rho_h^n\|_{L^1(E)},
\end{array}
\end{equation}
where we used \eqref{inv:H1-L2} and \eqref{sta_Ih_H1}.  Comparing \eqref{off-diagonal} with \eqref{lm5.1-lab1}, we find that
\begin{align*}
h^{\gamma}\int_E\nabla \varphi_{\a_i}\cdot\nabla \varphi_{\a_j}\dx &- \int_E\varphi_{\a_i} \nabla \mathcal{I}_h(K*[\rho_h^n]_T)\cdot\nabla\varphi_{\a_j}\dx
\\
&\le h^{d-2} h^{\gamma} (-C_{\rm neg} + C h^{1-\gamma}\|\nabla K\|_{L^{\infty}(\R^d)} \|\rho_h^n\|_{L^1(E)})<0
\end{align*}
holds if we let $ C h^{1-\gamma}\|\nabla K\|_{L^{\infty}(\R^d)} \|\rho_h^n\|_{L^1(E)}< C_{\rm neg}$, which is a consequence of \eqref{Restriction_III}. As a result, summing over $E\in {\rm supp }\, \varphi_{\a_i}\cap {\rm supp }\, \varphi_{\a_j}  $ yields
\begin{equation}\label{lm5.1-lab2}
h^\gamma (\nabla \varphi_{\a_i}, \nabla \varphi_{\a_j}) - (\varphi_{\a_i} \nabla\mathcal{I}_h (K*[\rho_h^n]_T),\nabla\varphi_{\a_j})<0.
\end{equation}
Analogously, we have, from \eqref{diagonal}, that
\begin{equation}\label{lm5.1-lab3}
h^\gamma(\nabla \varphi_{\a_i}, \nabla \varphi_{\a_i}) - (\varphi_{\a_i} \nabla\mathcal{I}_h (K*[\rho_h^n]_T),\nabla\varphi_{\a_i})>0
\end{equation}
holds if we let $ C h^{1-\gamma} \|\nabla K\|_{L^{\infty}(\R^d)}\|\rho_h^n\|_{L^1(E)}< C_{\rm neg}$, which is globally imposed in \eqref{Restriction_III}.

Now let $\rho_{h}^{\rm min}\in D_h$ be defined as
$$
\rho_{h}^{\rm min}=\sum_{\a\in\mathcal{N}_h} \rho_{h}^{-}(\a)\varphi_\a,
$$
where $\rho_{h}^{-}(\a)=\min\{0, \rho^{n+1}_{h}(\a)\}$. Analogously, one defines $\rho^{\rm max}_h\in D_h$ as
$$
\rho_{h}^{\rm max}=\sum_{\a\in\mathcal{N}_h} \rho_{h}^{+}(\a)\varphi_\a,$$
where $\rho_{h}^{+}(\a)=\max\{0, \rho^{n+1}_{h}(\a)\}$. Notice that $\rho^{n+1}_{h}=\rho_{h}^{\rm min}+ \rho_{h}^{\rm max}$. Set $\bar\rho_h=\rho^{\rm min}_h$ in \eqref{Scheme} to get
\begin{equation}\label{lm5.1-lab4}
\begin{array}{ll}
(\delta_t\rho^{n+1}_h,\rho^{\rm min}_h)_h&+ h^\gamma (\nabla\rho^{n+1}_h,\nabla \rho^{\rm min}_h)
\\
&+(\nabla\mathcal{I}_h  A([\rho^{n+1}_h]_T),\nabla \rho^{\rm min}_h)-(\rho^{n+1}_h\nabla\mathcal{I}_h( K*[\rho^n_h]_T),\nabla \rho^{\rm min}_h )=0.
\end{array}
\end{equation}
We will handle each term of \eqref{lm5.1-lab4} in order to show that $\rho_h^{\rm min}\equiv0$. Indeed, in virtue of the equality
$$
(\rho^{n+1}_h, \rho^{\rm min}_h)_h=(\rho^{\rm min}_h+\rho^{\rm max}_h, \rho^{\rm min}_h)_h=\|\rho^{\rm min}_h\|_h^2,
$$
it follows that
\begin{equation}\label{lm5.1-lab6}
(\delta_t\rho^{n+1}_h,\rho^{\rm min}_h)_h= \frac{1}{k}(\|\rho_h^{\rm min}\|_h^2-(\rho^n_h, \rho^{\rm min}_h))\ge \frac{1}{k} \|\rho_h^{\rm min}\|_h^2.
\end{equation}
Observe that we have
\begin{equation}\label{lm5.1-lab7}
\begin{array}{rcl}
(\nabla\mathcal{I}_h  A([\rho^{n+1}_h]_T),\nabla \rho^{\rm min}_h)&=&\displaystyle
\sum_{\a\not=\tilde\a\in \mathcal{N}_h} A([\rho_h^{n+1}(\a)]_T) \rho^{\rm min}_h(\tilde\a) (\nabla \varphi_\a, \nabla\varphi_{\tilde \a})
\\
&&\displaystyle
+\sum_{\a\in \mathcal{N}_h} A([\rho_h^{n+1}(\a)]_T) \rho^{\rm min}_h(\a) (\nabla \varphi_\a, \nabla\varphi_{\a})
\\
&=&\displaystyle
\sum_{\a\not=\tilde\a\in \mathcal{N}_h} A([\rho_h^{n+1}(\a)]_T) \rho^{\rm min}_h(\tilde\a) (\nabla \varphi_\a, \nabla\varphi_{\tilde \a})>0
\end{array}
\end{equation}
from  \eqref{off-diagonal} and $A([\rho_h^{n+1}(\a)]_T) \rho^{\rm min}_h(\tilde\a)\le0$. By the decomposition
\begin{align*}
h^\gamma (\nabla\rho^{n+1}_h,\nabla \rho^{\rm min}_h)&-(\rho^{n+1}_h\nabla\mathcal{I}_h( K*[\rho^n_h]_T),\nabla \rho^{\rm min}_h )
\\
=&h^\gamma (\nabla\rho^{\rm max}_h,\nabla \rho^{\rm min}_h)-(\rho^{\rm max}_h\nabla\mathcal{I}_h(K*[\rho^n_h]_T),\nabla \rho^{\rm min}_h )
\\
&+h^\gamma (\nabla\rho^{\rm min}_h,\nabla \rho^{\rm min}_h)-(\rho^{\rm min}_h\nabla \mathcal{I}_h(K*[\rho^n_h]_T),\nabla \rho^{\rm min}_h ),
\end{align*}
we deduce from \eqref{lm5.1-lab2} and \eqref{lm5.1-lab3}  that
\begin{align*}
\nonumber
h^\gamma (\nabla\rho^{\rm max}_h,&\nabla \rho^{\rm min}_h)-(\rho^{\rm max}_h\nabla\mathcal{I}_h (K*[\rho^n_h]_T),\nabla \rho^{\rm min}_h )
\\
=&\displaystyle
\sum_{\a\not=\tilde\a\in\mathcal{N}_h}\rho^{\rm max}_h(\a) \rho^{\rm min}_h(\tilde\a)\Big[ (h^\gamma (\nabla\varphi_\a,\nabla\varphi_{\tilde \a} )-(\varphi_\a\nabla\mathcal{I}_h(K*[\rho^n_h]_T),\nabla \varphi_{\tilde \a})\Big]
\\\nonumber&\displaystyle
+\sum_{\a\in\mathcal{N}_h} \rho^{\rm max}_h(\a) \rho^{\rm min}_h(\a)\Big[ (h^\gamma (\nabla\varphi_\a,\nabla\varphi_{\a})-(\varphi_\a\nabla \mathcal{I}_h(K*[\rho^n_h]_T),\nabla \varphi_{\a})\Big]\ge0
\end{align*}
since $\rho^{\rm max}_h(\a) \rho^{\rm min}_h(\tilde\a)\le0$ and $\rho^{\rm max}_h(\a) \rho^{\rm min}_h(\a)=0$. Therefore,
\begin{equation}\label{lm5.1-lab8}
\begin{array}{rcl}
h^\gamma (\nabla\rho^{\rm min}_h,\nabla \rho^{\rm min}_h)&-&(\rho^{\rm min}_h\mathcal{I}_h(\nabla K*[\rho^n_h]_T),\nabla \rho^{\rm min}_h)
\\
&\le& h^\gamma (\nabla\rho^{n+1}_h,\nabla \rho^{\rm min}_h)-(\rho^{n+1}_h\nabla \mathcal{I}_h(K*[\rho^n_h]_T),\nabla \rho^{\rm min}_h).
\end{array}
\end{equation}
As a result, we infer on applying \eqref{lm5.1-lab6}-\eqref{lm5.1-lab8}  into \eqref{lm5.1-lab4}  that
$$
\|\rho^{\rm min}_h\|^2_h+k\,h^\gamma \|\nabla\rho^{\rm min}_h\|^2\le k (\rho^{\rm min}_h\nabla \mathcal{I}_h(K*[\rho^n_h]_T),\nabla \rho^{\rm min}_h ).
$$
We know from \eqref{lm4.2-lab4} and \eqref{Equivalence-L2} that
$$
\|\rho^{\rm min}_h\|^2_h+k h^\gamma \|\nabla\rho^{\rm min}_h\|^2\le C k \| K\|_{W^{2,\infty}(\R^d)} \|\rho^n_h\|_{L^1(\Omega)} \|\rho^{\rm min}_h\|^2_h.
$$
Thus, from \eqref{Restriction_II},
$$
\|\rho_h^{\rm min}\|_h^2\le0,
$$
which implies that $\rho^{\rm min}_h\equiv0$ and hence  $\rho^{n+1}_h\ge 0$. It completes the proof.
\end{proof}
%\begin{remark} In case of $A(\rho)=\rho$, a discrete maximum principle can be attained, i.e., a pointwise upper bound for $\rho_h^{n+1}$. The technique follows  {\tt mutatis mutandis} the proof of nonnegativity.
%\end{remark}

Since we do not have a pointwise upper bound for $\rho^{n+1}_h$, we must slightly modify the argument leading to a priori energy estimates from \cite{Bertozzi_Slepcev_2010}, which uses the maximum principle. %This relation of a priori energy estimates with maximum principle  also appears in other models when constructing convergent numerical schemes; see, for instance, \cite{}.
\begin{lemma}[Energy estimates]\label{lm:Energy_Estimates} Assume that \eqref{Restriction_II-global} and \eqref{Restriction_III-global} are satisfied. Then  the sequence $\{\rho^{n}_h\}_{n=1}^{N}$ computed via scheme \eqref{Scheme} satisfies
\begin{equation}\label{lm5.1-est-L1}
\|\rho^{n+1}_h\|_{L^1(\Omega)}=\|\rho_h^0\|_{L^1(\Omega)}:=B_{L^1}
\end{equation}
and
\begin{equation}\label{lm5.1-est-L2}
\begin{array}{rcl}
\displaystyle
\|\rho^{n+1}_h\|^2_h+ \sum_{m=0}^{n}(k^2 \|\partial_t\rho^{n+1}_h\|^2_h+k\, h^\gamma\|\nabla\rho^{m+1}_h\|^2_{L^2(\Omega)}+k\|\nabla\mathcal{I}_hA_T(\rho^{m+1}_h)\|^2_{L^2(\Omega)})
\\
\le e^{T B_{L^1}\|K\|_{W^{2,\infty}(\R^d)}} \|\rho^0_h\|^2_h:=B^2_{L^2}.
\end{array}
\end{equation}
\end{lemma}
\begin{proof} We proceed by induction on $n$ to  prove \eqref{lm5.1-est-L1}. From \eqref{Initial-Positivity}, we know that $\rho^1_h\ge 0$ is true by Lemma~\ref{lm:Positivity} for \eqref{Restriction_II-global} and \eqref{Restriction_III-global}. On selecting $\bar\rho_h=1$ in \eqref{Scheme}, we obtain \eqref{lm5.1-est-L1} for $n=0$. The same argument leads us to proving that \eqref{lm5.1-est-L1} holds from $\rho^{n}_h\ge 0$ and $\|\rho^n_h\|_{L^1(\Omega)}=\|\rho_{0h}\|_{L^1(\Omega)}$ by induction hypothesis. At this point, it should be noted that \eqref{Restriction_II} and \eqref{Restriction_III}  combined with \eqref{lm5.1-est-L1} imply \eqref{Restriction_II-global} and \eqref{Restriction_III-global}.

Now let $\bar\rho_h=\rho^{n+1}_h$ in \eqref{Scheme} to get
\begin{align*}
\|\rho^{n+1}_h\|_h^2&+\|\rho_h^{n+1}-\rho_h^n\|_h^2+ 2 k h^\gamma \|\nabla\rho^{n+1}_h\|^2_{L^2(\Omega)}+ 2 C_{\rm Lip}^{-1} k\|\nabla \mathcal{I}_hA_T(\rho^{n+1}_h)\|^2_{L^2(\Omega)}
\\
&\le \|\rho^{n}_h\|_h^2+2 k (\rho^{n+1}_h\nabla \mathcal{I}_h(K*[\rho^n_h]_T),\nabla\rho_h^{n+1}),
\end{align*}
where we have used \eqref{Coercitivity} for $f=A_T$ being non-decreasing and Lipschitzian. Repeating the argument that led to estimating \eqref{lm4.2-lab4} and noting \eqref{lm5.1-est-L1} and \eqref{Equivalence-L2} yields
\begin{align*}
\|\rho^{n+1}_h\|_h^2&+\|\rho_h^{n+1}-\rho_h^n\|_h^2+ h^\gamma \|\nabla\rho^{n+1}_h\|_{L^2(\Omega)}^2+C_{\rm Lip}^{-1}\|\nabla \mathcal{I}_hA_T(\rho^{n+1}_h)\|^2_{L^2(\Omega)}
\\
&\le \|\rho^{n}_h\|_h^2+ C k \|K\|_{W^{2,\infty}(\R^d)} \|\rho^{0}_h\|_{L^1(\Omega)} \|\rho_h^{n+1}\|^2_h.
\end{align*}
By a discrete Grönwall lemma, we conclude that \eqref{lm5.1-est-L2} holds under condition \eqref{Restriction_II}.
\end{proof}

The constants $B_{L^1}$ and $B_{L^2}$ can be estimated uniformly with respect to $h$ in term of $\rho^0$ from \eqref{initial-boundness}.
\begin{corollary} It follows that
\begin{equation}\label{est:derivative}
k\sum_{n=0}^{N-1}\|\delta_t\rho^{n+1}_h\|^2_{(H^1(\Omega))'}\le C,
\end{equation}
where $C>0$ is a constant independent of $h$.
\end{corollary}
\begin{proof} Apply a standard duality technique to obtain  \eqref{est:derivative} from \eqref{Scheme} and \eqref{lm5.1-est-L2}.
\end{proof}
We end this section by summarizing the results of Lemmas \ref{lm:Positivity} and \ref{lm:Energy_Estimates}. Bounds \eqref{lm5.1-est-L2} and \eqref{est:derivative} yield that
\begin{equation}\label{Bound:positivity}
\rho_{h,k}, \rho_{h,k}^\pm\ge0 \quad\mbox{ in  }\quad Q,
\end{equation}
\begin{equation}\label{BoundLinfL2:rho_hk}
\{\rho_{h,k}\}_{h,k}, \{\rho_{h,k}^\pm\}_{h,k} \mbox{ are bounded in }  L^\infty(0,T; L^2(\Omega)),
\end{equation}
\begin{equation}\label{BoundL2H1:rho_hk}
\{ h^\frac{\gamma}{2} \rho_{h,k}^{+}\}_{h,k}, \{\mathcal{I}_hA([\rho_{h,k}^+]_T)\}_{h,k} \mbox{ are bounded in } L^2(0,T; H^1(\Omega)),
\end{equation}
and
\begin{equation}\label{Bound:der-rho_hk}
\{\rho_{h,k}\}_{h,k}\mbox{ is bounded in } H^1(0,T; (H^1(\Omega))')
\end{equation}
and, by passing to the limit in a subsequence, denoted by $(k,h)$ for convenience,  that there exists $\rho\in L^\infty(0,T; L^2(\Omega))$ such that
\begin{equation}\label{weak-convergence-rho}
\rho_{h,k}, \rho_{h,k}^\pm\to \rho\quad\mbox{ in }\quad L^\infty(0,T; L^2(\Omega))\mbox{-weakly}*,
\end{equation}
and
\begin{equation}\label{weak-convergence-derivative-rho}
\rho_{h,k}\to \rho\quad\mbox{ in }\quad H^1(0,T; (H^1(\Omega))')\mbox{-weakly}
\end{equation}
as $(h,k)\to (0,0)$. Moreover, there exists $\chi\in L^2(0,T, H^1(\Omega))$ such that
\begin{equation}\label{weak-convergence-A}
\mathcal{I}_h A([\rho_{h,k}^+]_T)\to \chi\quad\mbox{ in }\quad L^2(0,T; H^1(\Omega))\mbox{-weakly}
\end{equation}
as $(h,k)\to (0,0)$.

\section{Compactness}
As we are dealing with a nonlinear equation, the key ingredient in passing to the limit is obtaining  compactness of the discrete solutions computed using \eqref{Scheme}. Since we do not have control of the gradient of the discrete solutions due to the degenerate diffusion term, compactness turns out to be more complicated to achieve than in the non-degenerate case. We have split the proof into a series of four lemmas.
\begin{lemma}\label{lm:continuity_I} There exists a nonincreasing function $F_1: [0,\infty)\to [0,\infty)$ with $F_1(z)\to 0$ as $z\to 0$ such that for any sequence of  discrete solutions $\{\rho_{h,k}^+\}_{h,k}$ computed via scheme \eqref{Scheme} satisfies
\begin{equation}\label{lm:6.1-est:Omega}
\|[\rho^{+}_{h,k}]_T(t+\delta)-[\rho^{+}_{h,k}]_T(t)\|_{L^2(\Omega)}\le F_1(\|\mathcal{I}_hA([\rho^{+}_{h,k}]_T(t+\delta))-\mathcal{I}_hA([\rho^{+}_{h,k}]_T(t))\|_{L^2(\Omega)}),
\end{equation}
for all $\delta\in(0,T)$ and $t\in[0,T-\delta]$, and
\begin{align}\label{lm:6.1-est:omega}
\int_0^{T}  \int_{\omega}&|\mathcal{P}_h[\rho_{h,k}^+]_T(t,\x+\delta\boldsymbol{e}_i)-\mathcal{P}_h[\rho_{h,k}^+]_T(t,\x)|^2\dx\,\dt
\\
\nonumber
&\le F_1\left[\int_0^{T}  \int_{\omega}|\mathcal{P}_h\mathcal{I}_hA([\rho_{h,k}^+]_T(t,\x+\delta\boldsymbol{e}_i))-\mathcal{P}_h\mathcal{I}_hA([\rho_{h,k}^+]_T(t,\x))|^2\dx\,\dt \right],
\end{align}
for all $\omega \subset\subset\Omega$ and $0<\delta<{\rm dist}(\omega,\partial\Omega)$ with $\{\e_i\}_{i=1}^d$ being the Cartesian basis of $\R^d$.
\end{lemma}
\begin{proof} For $x\ge0$ and $y\ge0$,  define the following continuous function
$$
\sigma(x,y)=\left\{
\begin{array}{ccl}
\frac{A(x)-A(y)}{x-y}&\mbox{ if }& x\not=y,
\\
A'(x)&\mbox{ if }&x=y.
\end{array}
\right.
$$ 
Let $\eta>0$ and consider  $f(\eta)=\min\{\sigma (x,y): (x,y) \in [\eta, B_{L^\infty}]\times[0, B_{L^\infty}]\}$. 
Then we have that $f(\eta)>0$ since $A(x)'>0$ for $x>0$ owing to $\rm(A1)$. Let
$$
\mathcal{N}^1_h=\{\a\in\mathcal{N}_h: |[\rho_{h,k}^+]_T(t+\delta, \a)|<\eta\mbox{ and } |[\rho_{h,k}^+]_T(t, \a)|<\eta \}
$$
and
$\mathcal{N}^2_h=\mathcal{N}_h\backslash\mathcal{N}^1_h$. Then, from \eqref{Equivalence-L2}, we get
$$
\begin{array}{rcl}
\|[\rho_{h,k}^+]_T(t+\delta)-[\rho_{h,k}^+]_T(t)\|^2_{L^2(\Omega)}&\le& C_{\rm eq} \|[\rho_{h,k}^+]_T(t+\delta)-[\rho_{h,k}^+]_T(t)\|_h^2
\\
&=&
\displaystyle
C_{\rm eq}\sum_{\a\in\mathcal{N}^1_h} ([\rho_{h,k}^+]_T(t+\delta,\a)-[\rho_{h,k}^+]_T(t, \a))^2\int_\Omega\varphi_\a
\\
&&
\displaystyle
+C_{\rm eq}\sum_{\a\in\mathcal{N}^2_h} ([\rho_{h,k}^+]_T(t+\delta,\a)-[\rho_{h,k}^+]_T(t, \a))^2\int_\Omega\varphi_\a
\\
&\le&
\displaystyle
C_{\rm eq}\eta^2 |\Omega|
\\
&&\displaystyle
+C_{\rm eq}{f^{-2}(\eta)}\|\mathcal{I}_hA([\rho_{h,k}^+]_T(t+\delta))-\mathcal{I}_hA([\rho_{h,k}^+]_T(t))\|^2_{L^2(\Omega)}.
\end{array}
$$
Consider  $F_\mu(z)=\inf_{\eta>0} \{C_{\rm eq}|\Omega|\eta^2+C_{\rm eq}\mu{f^{-2}(\eta)}z^2\}$ with $\mu=1$ to complete the proof of \eqref{lm:6.1-est:Omega}.

For \eqref{lm:6.1-est:omega}, we reason along the same line as before. Define
$$
Q^1_{\omega}=\{ (t,\x)\in [0, T]\times\omega : |\mathcal{P}_h [\rho_{h,k}^+]_T(t,\x+\delta\e_i) |< \eta\mbox{ and } | \mathcal{P}_h[\rho_{h,k}^+]_T(t,\x) |< \eta  \}
$$
and $Q^2_{\omega}=Q\backslash Q^1_{\omega}$. Estimating as before, we find
\begin{align*}
\nonumber
\int_0^T \int_\omega& |\mathcal{P}_h [\rho_{h,k}^+]_T(t,\x+\delta\e_i)-\mathcal{P}_h[\rho_{h,k}^+]_T(t, \x)|^2
\\
\nonumber
=&\int_{Q^1_{\omega}} |\mathcal{P}_h [\rho_{h,k}^+]_T(t,\x+\delta\e_i)-\mathcal{P}_h[\rho_{h,k}^+]_T(t, \x)|^2 \dx\dt
\\
\label{lm6.1-lab3}
&+\int_{Q^2_{\omega}} |\mathcal{P}_h [\rho_{h,k}^+]_T(t,\x+\delta\e_i)-\mathcal{P}_h[\rho_{h,k}^+]_T)(t, \x)|^2 \dx\,\dt
%\\
%\le&\eta |Q|+\frac{1}{f^2(\eta)}\sum_{(\tau_n,E)\in Q_{h,k}^2} |\mathcal{P}_h\mathcal{I}_hA([\rho_{h,k}^+]_T(t_{\tau_n},\a_E+\delta\e_i))-\mathcal{P}_h\mathcal{I}_hA([\rho_{h,k}^+]_T(t_\tau, \a_E))|^2 |\tau_n| |E|
\\
\nonumber
\le&\eta |Q|+\frac{1}{f^2(\eta)} \int_0^T \int_\omega |\mathcal{P}_h\mathcal{I}_hA([\rho_{h,k}^+]_T(t,\x+\delta\e_i))-\mathcal{P}_h\mathcal{I}_hA([\rho_{h,k}^+]_T(t, \x))|^2\dx\,\dt, \end{align*}
%Again Minkowski's inequality gives
%\begin{align*}
%\int_0^T \int_\omega& |\mathcal{P}_h\mathcal{I}_hA([\rho_{h,k}^+]_T(t,\x+\delta\e_i))-\mathcal{P}_h\mathcal{I}_hA([\rho_{h,k}^+]_T(t, \x))|^2\dx\dt
%\\
%&\le C  \int_0^T  \|\mathcal{I}_hA([\rho_{h,k}^+]_T-\mathcal{P}_h\mathcal{I}_hA([\rho_{h,k}^+]_T)\|^2 \dt
%\\
%&+ C \int_0^T \int_\omega |\mathcal{I}_hA([\rho_{h,k}^+]_T(t,\x+\delta\e_i))-\mathcal{I}_hA([\rho_{h,k}^+]_T(t, \x))|^2\dx\dt.
%\end{align*}
%and hence, by \eqref{error-Ph} and \eqref{lm5.1-est-L2},
%\begin{align*}
%\int_0^T \int_\omega& |\mathcal{P}_h\mathcal{I}_hA([\rho_{h,k}^+]_T(t,\x+\delta\e_i))-\mathcal{P}_h\mathcal{I}_hA([\rho_{h,k}^+]_T(t, \x))|^2\dx\dt
%\\
%&\le C h + C \int_0^T \int_\omega |\mathcal{I}_hA([\rho_{h,k}^+]_T(t,\x+\delta\e_i))-\mathcal{I}_hA([\rho_{h,k}^+]_T(t, \x))|^2\dx\dt.
%\end{align*}
which implies \eqref{lm:6.1-est:omega}. 
\end{proof}
\begin{lemma}\label{lm:continuity_II} Let $\delta\in (0,T)$ and $t\in[0,T-\delta]$.  Assume that there exists $B>0$ such that the sequence of discrete solutions $\{\rho_{h,k}^+\}_{h,k}$ computed via \eqref{Scheme} satisfies
\begin{equation}\label{lm6.2-lab1}
\|\mathcal{I}_hA ([\rho_{h,k}^+]_T(t+\delta))\|_{H^1(\Omega)}\le B\quad\mbox{ and }\quad \|\nabla\mathcal{I}_hA ([\rho_{h,k}^+]_T(t))\|_{H^1(\Omega)}\le B.
\end{equation}
and
\begin{equation}\label{lm6.2-lab2}
h^\gamma\|\nabla\rho_{h,k}^+(t+\delta)\|_{L^2(\Omega)}\le B\quad\mbox{ and }\quad h^\gamma\|\nabla\rho_{h,k}^+(t))\|_{L^2(\Omega)}\le B.
\end{equation}
Then there exists a function $G_B :[0,\infty)\to [0,\infty)$ being nondecreasing and satisfying $G_B(\varepsilon)\to 0$ as $\varepsilon\to 0$ such that
$$
\|[\rho_{h,k}^+(t+\delta)]_T-[\rho_{h,k}^+(t)]_T\|^2_{L^2(\Omega)}\le G_B(\varepsilon)
$$
providing that $$
([\rho_{h,k}^+]_T(t+\delta)-[\rho_{h,k}^+]_T(t), \mathcal{I}_h A([\rho_{h,k}^+]_T(t+\delta))-\mathcal{I}_hA([\rho_{h,k}^+]_T(t)))_h\le \varepsilon
$$
holds
\end{lemma}
\begin{proof}
We establish the lemma by contradiction. Assume that there exist $\kappa>0$ and two sequences $\{\rho_{h_n,k_n}^+(t+\delta)\}_{n=1}^\infty$ and $\{\rho_{h_n,k_n}^+(t)\}_{n=1}^\infty$ such that
\begin{equation}\label{lm6.2-lab3}
([\rho_{h_n,k_n}^+]_T(t+\delta)-[\rho_{h_n,k_n}^+]_T(t), \mathcal{I}_{h_n} A([\rho_{h_n,k_n}^+]_T(t+\delta))-\mathcal{I}_{h_n}A([\rho_{h_n,k_n}^+]_T(t)))_{h_n}\le\frac{1}{n}
\end{equation}
and
\begin{equation}\label{lm6.2-lab4}
\|[\rho_{h_n,k_n}^+]_T(t+\delta)-[\rho_{h_n,k_n}^+]_T(t)\|^2_{L^2(\Omega)}>\kappa.
\end{equation}
From \eqref{lm6.2-lab1}, we know that there exist $w_1,w_2\in L^2(\Omega)$ and a subsequence of $\{[\rho_{h_n,k_n}^+]_T(t+\delta)\}_{n=0}^\infty$ and  $\{ [\rho_{h_n,k_n}^+]_T(t)\}_{n=0}^\infty$, still denoted by itself, such that
$$
\mathcal{I}_{h_n} A([\rho_{h_n,k_n}^+]_T(t+\delta)) \to A(\rho_1)\quad\mbox{ in } \quad L^2(\Omega)\quad \mbox{ as }\quad n\to \infty
$$
and
$$
\mathcal{I}_{h_n} A([\rho_{h_n,k_n}^+]_T(t)) \to A(\rho_2) \quad\mbox{ in } \quad L^2(\Omega)\quad \mbox{ as }\quad n\to \infty,
$$
where $\rho_1=A^{-1}(w_1)$ and $\rho_2=A^{-1}(w_2)$. It is not hard to see from \eqref{error-Ph} and \eqref{lm6.2-lab1} that
$$
\mathcal{P}_{h_n}\mathcal{I}_{h_n} A([\rho_{h_n,k_n}^+]_T(t+\delta)) \to A(\rho_1)\quad\mbox{ in } \quad L^2(\Omega)\quad \mbox{ as }\quad n\to \infty
$$
and
$$
\mathcal{P}_{h_n}\mathcal{I}_{h_n} A([\rho_{h_n,k_n}^+]_T(t)) \to A(\rho_2) \quad\mbox{ in } \quad L^2(\Omega)\quad \mbox{ as }\quad n\to \infty.
$$
Lebesgue's Dominated Convergence Theorem  implies that
$$
A^{-1}\mathcal{P}_{h_n} \mathcal{I}_{h_n} A([\rho_{h_n,k_n}^+]_T(t+\delta)) \to\rho_1\quad\mbox{ in }\quad L^2(\Omega)\quad\mbox{ as }\quad n\to\infty
$$
and
$$
A^{-1} \mathcal{P}_{h_n} \mathcal{I}_{h_n}A([\rho_{h_n,k_n}^+]_T(t))\to \rho_2\quad\mbox{ in }\quad L^2(\Omega)\quad\mbox{ as }\quad n\to\infty.
$$

It is clear that
$$
\mathcal{P}_{h_n}[\rho_{h_n,k_n}^+]_T(t+\delta)= A^{-1}\mathcal{P}_{h_n}\mathcal{I}_h A([\rho_{h_n,k_n}^+]_T(t+\delta))
$$
and
$$
\mathcal{P}_{h_n} [\rho_{h_n,k_n}^+]_T(t)= A^{-1}\mathcal{P}_{h_n}\mathcal{I}_{h_n} A([\rho_{h_n,k_n}^+]_T(t)).
$$
In view of \eqref{error-Ph} and \eqref{lm6.2-lab2}, we get
$$
\|[\rho_{h_n,k_n}^+]_T(t+\delta)-\mathcal{P}_h\rho_{h_n,k}^+]_T(t+\delta)\|_{L^2(\Omega)}\le C h_n \|\nabla \rho_{h_n,k_n}^+(t+\delta)\|_{L^2(\Omega)}\le C h_n^{1-\gamma} B
$$
and
$$
\|\rho_{h_n,k_n}^+(t)-\mathcal{P}_h\rho_{h_n,k_n}^+(t)\|_{L^2(\Omega)}\le C h_n \|\nabla \rho_{h_n,k_n}^+(t) \|_{L^2(\Omega)}\le C h_n^{1-\gamma} B.
$$
Here, we used the fact that $|\nabla [\rho^{+}_{h_n,k_n}]_T(\cdot)|\le |\nabla [\rho^{+}_{h_n,k_n}]_T(\cdot)| $. Therefore,
$$
[\rho_{h_n,k_n}^+]_T(t+\delta)\to\rho_1\quad\mbox{ and } \quad [\rho_{h_n,k_n}^+]_T(t)\to \rho_2\quad\mbox{ in }\quad L^2(\Omega)\quad\mbox{ as }\quad n\to \infty.
$$
On noting  \eqref{error-L1-L2-H1}, we have
$$
\begin{array}{l}
([\rho_{h_n,k_n}^+]_T(t+\delta)-[\rho_{h_n,k_n}^+]_T(t), \mathcal{I}_{h_n} A([\rho_{h_n,k_n}^+]_T(t+\delta))-\mathcal{I}_{h_n}A([\rho_{h_n,k_n}^+]_T(t)))
\\\displaystyle
\le\frac{1}{n}+ C h_n \|[\rho_{h_n,k_n}^+]_T(t+\delta)-[\rho_{h_n,k_n}^+]_T(t)\|_{L^2(\Omega)}\|\nabla( \mathcal{I}_{h_n} (A([\rho_{h_n,k_n}^+]_T(t+\delta))-A([\rho_{h_n,k_n}^+]_T(t))))\|_{L^2(\Omega)}
\\
\le\displaystyle
\frac{1}{n}+C |\Omega|^\frac{1}{2} h_n B_{L^\infty} B.
\end{array}
$$
Passing to the limit in this last estimate yields
$$
(\rho_1-\rho_2, A(\rho_1)-A(\rho_1))=0,
$$
which implies that $\rho_1=\rho_2$. As a result, we have $\|\mathcal{I}_{h_n} A([\rho_{h_n,k_n}]_T(t+\delta))-\mathcal{I}_hA([\rho_{h_n,k_n}]_T(t))\|_{L^2(\Omega)}\to 0$ as $n\to +\infty$. But then $\|[\rho_{h_n,k_n}]_T(t+\delta)-[\rho_{h_n,k_n}]_T(t)\|_{L^2(\Omega)}\to 0$ as $n\to +\infty$ from Lemma \ref{lm:continuity_I}, which is a contradiction from \eqref{lm6.2-lab4}.
\end{proof}

In order to prove the following lemma, we draw on \cite[Prop. 27]{GG_GS_2008}.
\begin{lemma} Let $\delta\in(0,T)$ and $t\in[0,T-\delta]$. Then it follows that
\begin{equation}\label{lm:6.3-bound}
\int_0^{T-\delta}([\rho_{h,k}^+]_T(t+\delta)-[\rho_{h,k}]_T^+(t),  A([\rho_{h,k}^+]_T(t+\delta))- A([\rho_{h,k}^+]_T(t)))_h\, \dt\le C \delta.
\end{equation}
\end{lemma}
\begin{proof} Since $\rho_{h,k}^+$ is a time--stepping function, we only need to consider $\delta= r k$, with $r=1,\cdots, N$, and prove
$$
k\sum_{m=0}^{N-r} ([\rho_h^{m+r}]_T-[\rho_h^m]_T,  A([\rho_h^{m+r}]_T)- A([\rho_h^m]_T))_h\, \dt\le C (r k)^\frac{1}{2}.
$$
Let us test \eqref{Scheme} against $\bar\rho_h=\mathcal{I}_hA([\rho_h^{m+r}]_T)-\mathcal{I}_hA([\rho_h^m]_T)$ to obtain
\begin{align*}
(\rho^{n+1}_h-\rho_h^n&, A([\rho_h^{m+r}]_T)-A([\rho_h^m]_T))_h
\\
&=-h^\gamma k (\nabla \rho^{n+1}_h,\nabla\mathcal{I}_hA([\rho_h^{m+r}]_T)-\nabla\mathcal{I}_hA([\rho_h^m]_T))
\\
&-k(\nabla \mathcal{I}_h A([\rho_h^{n+1}]_T), \nabla \mathcal{I}_hA([\rho_h^{m+r}]_T)-\nabla \mathcal{I}_hA([\rho_h^m]_T))
\\
&+ k (\rho^{n+1}_h\nabla\mathcal{I}_h (K*[\rho_h^n]_T), \nabla \mathcal{I}_hA([\rho_h^{m+r}]_T)-\nabla\mathcal{I}_hA([\rho_h^m]_T).
\end{align*}
Summing for $n = m, ..., m - 1 + r$, multiplying by $k$ and summing for $m = 0, \cdots, N - r$  yields
\begin{align*}
 k\sum_{n=0}^{N-r}(\rho^{m+r}_h-\rho_h^m,& A([\rho_h^{m+r}]_T)-A([\rho_h^m]_T))_h
 \\
& =-h^\gamma k\sum_{n=0}^{N-r}k\sum_{n=m}^{m-1+r} (\nabla \rho^{n+1}_h,\nabla\mathcal{I}_hA([\rho_h^{m+r}]_T)-\nabla\mathcal{I}_hA([\rho_h^m]_T))
\\
&- k\sum_{n=0}^{N-r}k\sum_{n=m}^{m-1+r}(\nabla \mathcal{I}_h A([\rho_h^{n+1}]_T), \nabla \mathcal{I}_hA([\rho_h^{m+r}]_T)-\nabla \mathcal{I}_hA([\rho_h^m]_T))
\\
&+ k\sum_{n=0}^{N-r}k\sum_{n=m}^{m-1+r} (\rho^{n+1}_h\nabla \mathcal{I}_h(A*[\rho_h^n]_T), \nabla \mathcal{I}_hA([\rho_h^{m+r}]_T)-\nabla\mathcal{I}_hA([\rho_h^m]_T).
\end{align*}
We now proceed to bound each term on the right-hand side. In doing so, we first apply a Fubini discrete rule to write
$$
\begin{array}{l}
\displaystyle
h^\gamma k\sum_{n=0}^{N-r}k\sum_{n=m}^{m-1+r} (\nabla \rho^{n+1}_h,\nabla\mathcal{I}_hA([\rho_h^{m+r}]_T)-\nabla\mathcal{I}_hA([\rho_h^m]_T))
\\
\displaystyle
=h^\gamma k\sum_{n=0}^{N-1}k\sum_{m=\overline{n-1+r}}^{\bar n} (\nabla \rho^{n+1}_h,\nabla\mathcal{I}_hA([\rho_h^{m+r}]_T)-\nabla\mathcal{I}_hA([\rho_h^m]_T)),
\end{array}
$$
where
$$
\bar n=\left\{
\begin{array}{lcl}
0&\mbox{ if }& n<0,
\\
n&\mbox{ if }& 0\le n \le N-r,
\\
N-r&\mbox{ if }& n>N-r.
\end{array}
\right.
$$
Therefore, using $|\bar n-\overline{n-r+1}|\le r$, we have, by \eqref{lm5.1-est-L2}, that
$$
\begin{array}{l}
\displaystyle
h^\gamma k\sum_{n=0}^{N-r}k\sum_{n=m}^{m-1+r} (\nabla \rho^{n+1}_h,\nabla\mathcal{I}_hA([\rho_h^{m+r}]_T)-\nabla\mathcal{I}_hA([\rho_h^m]_T))
\\
\displaystyle
\le h^{\gamma^{\frac{1}{2}}} k\sum_{n=0}^{N-r} h^{\gamma^\frac{1}{2}} k \|\nabla \rho^{n+1}_h\|_{L^2(\Omega)} \left(\sum_{m=\overline{n-1+r}}^{\bar n} k\|\nabla(\mathcal{I}_hA([\rho_h^{m+r}]_T)-\mathcal{I}_hA([\rho_h^m]_T))\|^2_{L^2(\Omega)}\right)^\frac{1}{2} \left(\sum_{m=\overline{n-1+r}}^{\bar n} k\right)^\frac{1}{2}
\\
\le C B^2_{L^2}  T^\frac{1}{2} (rk)^\frac{1}{2}.
\end{array}
$$
Analogously, we bound
$$
k\sum_{n=0}^{N-r}k\sum_{n=m}^{m-1+r}(\nabla\mathcal{I}_hA([\rho_h^{n+1}]_T),\mathcal{I}_hA([\rho_h^{m+r}]_T)-\mathcal{I}_hA([\rho_h^m]_T))\le C T^\frac{1}{2} B^2_{L^2} (rk)^\frac{1}{2}
$$
and
\begin{align*}
k\sum_{n=0}^{N-r}k\sum_{n=m}^{m-1+r}& (\rho^{n+1}_h\nabla \mathcal{I}_h(K*[\rho_h^n]_T), \nabla \mathcal{I}_hA([\rho_h^{m+r}]_T)-\mathcal{I}_hA([\rho_h^m]_T)
\\
&\le C T \|K\|_{W^{2,\infty}(\R^d)} \|\rho_h^0\|_{L^1(\Omega)} B^2_{L^2} (rk)^\frac{1}{2}.
\end{align*}
Combining these above estimates gives
$$
k\sum_{n=0}^{N-r}(\rho^{m+r}_h-\rho_h^m, A([\rho_h^{m+r}]_T)-A([\rho_h^m]_T))_h\le C(rk)^\frac{1}{2}.
$$
The proof is now completed on noting that $|[\rho^{m+r}_h]_T(\a)-[\rho_h^m(\a)]_T|\le |\rho^{m+r}_h(\a)-\rho_h^m(\a)|$ for all $\a\in\mathcal{N}_h$.
\end{proof}
In order to set out that  the sequence of $\{[\rho_{h,k}^+]_T\}_{h,k}$ is precompact in $L^2(Q)$, we will use the Riesz-Fréchet-Kolmogorov compactness criterion.
\begin{lemma}\label{Lm:compactaness} It follows that
\begin{equation}\label{strong-convergence-rho-I}
[\rho_{h,k}^+]_T\to \rho_T\mbox{ in } L^2(Q)\mbox{-strongly as } (h,k)\to (0,0),
\end{equation}
where $\rho_T$ is the truncating of the limiting function $\rho$ obtained from the weak convergences.
\end{lemma}
\begin{proof} The proof should be understood for the subsequence obtained in \eqref{weak-convergence-rho}  and  \eqref{weak-convergence-derivative-rho}.  We divide the proof into two parts:

{\bf Part I}: We claim that for each $\varepsilon>0$ there exists $0<\delta_0\le T$ such that for all $(h,k)>0$ and all $0 < \delta < \delta_0$
\begin{equation}\label{lm6.4-lab1}
\int_0^{T-\varepsilon}  \|[\rho_{h,k}^+]_T(\cdot+\delta)-[\rho_{h,k}]_T^+(\cdot)\|^2_{L^2(\Omega)}\dt< \varepsilon.
\end{equation}

By Lemma \ref{lm:Energy_Estimates}, we know that
$$
\|A([\rho_{h,k}^+]_T)\|_{L^2(0,T; H^1(\Omega))}\le ( B_{L^\infty}^2 |\Omega|+ B_{L^2}^2)^{\frac{1}{2}}:=B.
$$
Consider $0<\delta<\varepsilon$ and $\theta>1$ and define
\begin{align*}
E_\theta(\delta)=\Big\{t\in [0,T-\varepsilon]:&\|A([\rho_{h,k}^+]_T(t))\|_{H^1(\Omega)}\le B \theta^{\frac{1}{2}}, \quad\|A([\rho_{h,k}^+]_T(t+\delta))\|^2_{H^1(\Omega)}\le B \theta^{\frac{1}{2}},
\\
&h^{\gamma}\|\nabla\rho_{h,k}^+(t)\|_{L^2(\Omega)}\le B\theta^\frac{1}{2},\quad h^{\gamma}\|\nabla\rho_{h,k}^+(t+\delta)\|^2_{L^2(\Omega)}\le B\theta^\frac{1}{2},
\\
&([\rho_{h,k}^+]_T(t+\delta)-[\rho_{h,k}^+]_T(t), A([\rho_{h,k}^+]_T(t+\delta))-A([\rho_{h,k}^+]_T(t)))_h\le C  \theta \delta\Big\}.
\end{align*}
By Chebyshev's inequality, we deduce that $|E^c_\theta(\delta)|\le \frac{5}{\theta}$, where $E^c_\theta(\delta)$ is the complementary set of $E_\theta(\delta)$. Therefore, by Lemma \ref{lm:continuity_II} combined with \eqref{lm:6.3-bound},
$$
\int_0^{T-\varepsilon}  \|[\rho_{h,k}^+]_T(\cdot+\delta)-[\rho_{h,k}^+]_T(\cdot)\|^2_{L^2(\Omega)}\dt\le T G_{B \theta^\frac{1}{2}}(C\theta\delta)+2B^2_{L^\infty}\frac{5}{\theta}.
$$
On choosing $\theta=\max\{\frac{20B^2_{L^\infty}}{\varepsilon}, 1\}$ and $\delta_0>0$ such that $T G_{ B \theta^\frac{1}{2}}(C\theta\delta)<\frac{\varepsilon}{2}$, this leads to \eqref{lm6.4-lab1}.

{\bf Part II}: We claim that for each $\varepsilon>0$ and each $\omega\subset\subset\Omega$ there exists $0<\delta_0< {\rm dist}(\omega,\partial\Omega)$ such that
\begin{equation}\label{lm6.4-lab2}
\int_0^{T}  \int_{\omega}|[\rho_{h,k}^+]_T(t,\x+\delta\boldsymbol{e}_i))-[\rho_{h,k}^+]_T(t,\x))|^2\dx\,\dt< \varepsilon,
\end{equation}
for all $(h,k)>0$, and all $0 < \delta \le \delta_0$ and $i=1,\cdots, d$.

Using Minkowski's inequality, we have
\begin{align*}
\nonumber
\int_0^T \int_\omega |[\rho_{h,k}^+]_T(t,\x+\delta\e_i)&-[\rho_{h,k}^+]_T(t, \x)|^2\dx\,\dt
\\
&\le C  \int_0^T  \|[\rho_{h,k}^+]_T-\mathcal{P}_h[\rho_{h,k}^+]_T\|^2_{L^2(\Omega)} \dt
\\
\nonumber
&+ C \int_0^T \int_\omega |\mathcal{P}_h [\rho_{h,k}^+]_T(t,\x+\delta\e_i)-\mathcal{P}_h[\rho_{h,k}^+]_T(t, \x)|^2\dx\,dt.
\end{align*}
We estimate each term on the right-hand side separately.  We have, by \eqref{error-Ph} and \eqref{lm5.1-est-L2},  that
\begin{equation*}
\int_0^T  \|[\rho_{h,k}^+]_T-\mathcal{P}_h[\rho_{h,k}^+]_T\|^2_{L^2(\Omega)} \dt\le C B_{L^2}^2 h^{1-\gamma},
\end{equation*}
where we have used that that fact $|\nabla[\rho_h]_T|\le |\nabla\rho_h|$ for all $\rho_h\in D_h$.

Now we want to use \eqref{lm:6.1-est:omega} to control the second term. Observe that
\begin{align*}
\nonumber
\int_0^T \int_\omega &|\mathcal{P}_h\mathcal{I}_h A([\rho_{h,k}^+]_T(t,\x+\delta\e_i))-\mathcal{P}_h\mathcal{I}_h A([\rho_{h,k}^+]_T(t, \x))|^2_{L^2(\Omega)}\dx\,\dt
\\
\le& C  \int_0^T  \|\mathcal{I}_h A([\rho_{h,k}^+]_T)-\mathcal{P}_h\mathcal{I}_h A([\rho_{h,k}^+]_T)\|^2_{L^2(\Omega)} \dt
\\
\nonumber
&+ C \int_0^T \int_\omega |\mathcal{I}_h A ([\rho_{h,k}^+]_T(t,\x+\delta\e_i))-\mathcal{I}_h A([\rho_{h,k}^+]_T(t, \x))|^2\dx\,dt.
\end{align*}
It is easily to check, from \eqref{error-Ph}, that
$$
\int_0^T  \|\mathcal{I}_h A([\rho_{h,k}^+]_T)-\mathcal{P}_h\mathcal{I}_h A([\rho_{h,k}^+]_T)\|^2_{L^2(\Omega)} \dt\le C B_{L^2}^2 h^2
$$
and, by the Mean-Value Theorem,  that
\begin{align*}
\int_0^{T}  \int_{\omega}|\mathcal{I}_hA([\rho_{h,k}^+]_T(t,\x+\delta\boldsymbol{e}_i)&-\mathcal{I}_hA([\rho_{h,k}^+]_T(t,\x))|^2\dx\,\dt
\\
&\le \delta^2 \int_0^T \|\nabla\mathcal{I}_hA([\rho_{h,k}^+(t)]_T)\|^2_{L^2(\Omega)}\dt\le\delta^2 B^2_{L^2}.
\end{align*}
Thus, by Lemma \ref{lm:continuity_I},
$$
\int_0^T \int_\omega |\mathcal{P}_h [\rho_{h,k}^+]_T(t,\x+\delta\e_i)-\mathcal{P}_h[\rho_{h,k}^+]_T(t, \x)|^2\dx\,\dt\le F_1( B^2_{L^2} (C h^2+\delta^2)).
$$

Therefore,
$$
\int_0^{T}  \int_{\omega}|\rho_{h,k}^+(t,\x+\delta\boldsymbol{e}_i)-\rho_{h,k}^+(t,\x)|^2\dx\,\dt\le C B_{L^2} h^{1-\gamma}+ F_2(\delta^2 B^2_{L^2})+F_2(C B^2_{L^2}  h^2). $$
Following the proof of \cite[Thm. 5.1]{Arzela_GG} we infer that \eqref{lm6.4-lab2} holds.

Finally, inequalities \eqref{lm6.4-lab1} and \eqref{lm6.4-lab2} are sufficient to prove that the sequence of $\{[\rho_{h,k}]_T\}_{h,k}$ is precompact via the Riesz-Fréchet-Kolmogorov compactness criterion. It is not hard to see that the limiting function $\rho_T$ is the truncating of $\rho$.
\end{proof}
We further infer that 
\begin{equation}\label{strong-convergence-rho-II}
[\rho_{h,k}^-]_T\to \rho_T\mbox{ in } L^2(Q)\mbox{-strongly as }\quad (h,k)\to (0,0).
\end{equation}

As a consequence of Lemma \ref{Lm:compactaness}, we have the following.
\begin{corollary} There holds
\begin{equation}\label{weak-convergence-A(rho)}
\mathcal{I}_hA([\rho_{h,k}^+]_T)\to A(\rho_T)\quad\mbox{ in }\quad L^2(0,T; H^1(\Omega))\mbox{-weakly as }\quad (h,k)\to (0,0).
\end{equation}
\end{corollary}
\begin{proof} Using Minkowski's inequality a few times, we see that
$$
\begin{array}{rcl}
\displaystyle
\int_0^T\|\mathcal{I}_hA([\rho_{h,k}^+]_T)-A(\rho_T)\|^2_{L^2(\Omega)}\dt&\le&\displaystyle
C \int_0^T\|\mathcal{I}_hA([\rho_{h,k}^+]_T)-\mathcal{P}_h\mathcal{I}_hA([\rho_{h,k}^+]_T)\|^2_{L^2(\Omega)}\dt
\\
&&\displaystyle
+C\int_0^T\|\mathcal{P}_h\mathcal{I}_hA([\rho_{h,k}^+]_T)-A([\rho_{h,k}^+]_T)\|^2_{L^2(\Omega)}\dt
\\
&&\displaystyle
+C \int_0^T\|A([\rho_{h,k}^+]_T)-A(\rho_T)\|^2_{L^2(\Omega)}\dt.
\end{array}
$$
In view of \eqref{error-Ph} and \eqref{BoundL2H1:rho_hk}, we obtain
$$
\int_0^T\|\mathcal{I}_hA([\rho_{h,k}^+]_T)-\mathcal{P}_h\mathcal{I}_hA([\rho_{h,k}^+]_T)\|^2_{L^2(\Omega)}\dt \le C\,h^2 \int_0^T\| \nabla\mathcal{I}_hA([\rho_{h,k}^+]_T)\|^2_{L^2(\Omega)}\dt \to 0
$$
and
$$
\begin{array}{rcl}
\displaystyle
\int_0^T\|\mathcal{P}_h\mathcal{I}_hA([\rho_{h,k}^+]_T)-A([\rho_{h,k}^+]_T)\|^2_{L^2(\Omega)}\dt&=&\displaystyle
\int_0^T\|\mathcal{P}_hA([\rho_{h,k}^+]_T)-A([\rho_{h,k}^+]_T)\|^2_{L^2(\Omega)}\dt
\\
&\le&\displaystyle
C h^2 \int_0^T \|\nabla A([\rho_{h,k}^+]_T)\|^2_{L^2(\Omega)}\dt
\\
&\le&\displaystyle
C h^2 \|A'([\rho_{h,k}^+]_T)\|_{L^\infty(Q)} \int_0^T \|\nabla [\rho_{h,k}^+]_T\|^2_{L^2(\Omega)}\dt
\\
&\le&\displaystyle
C h^{2-\gamma} A'(B_{L^\infty}) \int_0^T h^\gamma \|\nabla\rho_{h,k}^+\|^2_{L^2(\Omega)}\dt\to 0
\end{array}
$$
as $(h,k)\to (0,0)$. Lebesgue's Dominated Convergence
Theorem combined with \eqref{strong-convergence-rho-I} provides
$$
\int_0^T\|A([\rho_{h,k}^+]_T)-A(\rho_T)\|^2_{L^2(\Omega)}\dt\to 0
$$
as $(h,k)\to (0,0)$. Thus, the above convergence gives
$$
\mathcal{I}_hA([\rho_{h,k}^+]_T)\to A(\rho_T)\quad\mbox{ in }\quad L^2(Q)\mbox{-strongly as }\quad (h,k)\to (0,0).
$$
Furthermore, it follows from \eqref{weak-convergence-A} that \eqref{strong-convergence-rho-I} is satisfied; thus completing the proof.
\end{proof}
\section{Passage to the limit }\label{sec:Passage}
We briefly outline the main steps of the passage to the limit since the arguments are quite standard.

Let $\bar\rho\in L^2(0,T;W^{1,\infty}(\Omega))$. We know that $\mathcal{SZ}_h\bar\rho\to \bar\rho$ in $L^2(0,T;W^{1,\infty}(\Omega))$ as $h\to 0$ from \eqref{error-SZ}. Then  selecting $\bar\rho_h=\mathcal{SZ}_h\bar\rho$ in \eqref{Scheme}, multiplying by $k$, and summing over $n$ yields
\begin{align*}
&\int_0^T (\partial_t\rho_{h,k},\mathcal{SZ}_h\bar\rho)_h \dt+ \int_0^T h^\gamma (\nabla\rho_{h,k}^+,\nabla \mathcal{SZ}_h\bar\rho) \dt
\\
&+\int_0^T (\nabla\mathcal{I}_h  A([\rho_{k,h}^+]_T),\nabla \mathcal{SZ}_h\bar\rho)\dt-\int_0^T(\rho_{h,k}^+\nabla \mathcal{I}_h((K*[\rho_{h,k}^-]_T),\nabla\mathcal{SZ}_h\bar\rho)\,\dt=0.
\end{align*}
\begin{itemize}
\item For the time derivative, we have:
$$
\int_0^T (\partial_t\rho_{h,k},\mathcal{SZ}_h\rho)_h\dt=\displaystyle
\int_0^T [(\partial_t\rho_{k,h},\mathcal{SZ}_h\rho)_h-(\partial_t\rho_{h,k},\mathcal{SZ}_h\rho)]\dt+\int_0^T(\partial_t\rho_{h,k},\mathcal{SZ}_h\rho)\dt.
$$
It is clear from \eqref{weak-convergence-derivative-rho}  that
$$
\int_0^T(\partial_t\rho_{h,k}, \mathcal{SZ}_h\rho)\dt\to \int_0^T<\partial_t\rho,\bar\rho>\dt.
$$
To control the residual term, we use \eqref{error-L1-H-1-Linf} combined with \eqref{Bound:der-rho_hk} to see
\begin{align*}
\int_0^T& [(\delta_t\rho^{n+1}_h,\mathcal{SZ}_h\rho)_h-(\delta_t\rho^{n+1}_h,\mathcal{SZ}_h\rho)]\dt
\\
&\le C h^{\frac{1}{2}} \left(\int_0^T \|\partial_t\rho_{h,k}\|^2_{(H^1(\Omega))'}\dt\right)^{\frac{1}{2}} \left(\int_0^T\|\nabla\mathcal{SZ}_h\rho\|_{L^\infty(\Omega)}^2\dt\right)^\frac{1}{2} \to 0.
\end{align*}
\item For the dissipation terms, we have by \eqref{BoundL2H1:rho_hk} and \eqref{weak-convergence-A(rho)} that
$$
h^\gamma\int_0^T (\nabla\rho_{h,k}^+,\nabla\mathcal{SZ}_h\bar\rho)\dt\to0
$$
and
$$
\int_0^T (\nabla \mathcal{I}_hA([\rho_{h,k}^+]_T),\nabla\mathcal{SZ}_h\bar\rho)\dt\to\int_0^T (\nabla \mathcal{I}_hA(\rho_T),\nabla\bar\rho)\dt
$$
\item For the convolution term, we proceed as follows.
$$
\begin{array}{rcl}
\displaystyle
\int_0^T(\rho_{h,k}^+\nabla\mathcal{I}_h( K*[\rho_{h,k}^-]_T),\nabla\mathcal{SZ}_h\bar\rho)\dt&=&\displaystyle
\int_0^T(\rho_{h,k}^+\nabla  (\mathcal{I}_h-\mathcal{I})(K*[\rho_{h,k}^-]_T,)\nabla\mathcal{SZ}_h\bar\rho)\dt
\\
&&+\displaystyle
\int_0^T(\rho_{h,k}^+\nabla K*[\rho_{h,k}^-]_T,\nabla\mathcal{SZ}_h\bar\rho)\dt.
\end{array}
$$
For the first term, we have
\begin{align*}
\int_0^T&(\rho_{h,k}^+\nabla (\mathcal{I}_h-\mathcal{I})(K*[\rho_{h,k}^-]_T),\nabla\mathcal{SZ}_h\bar\rho)\,\dt
\\
&\le C h  \|\nabla^2 K\|_{L^{\infty}(\R^d)} \|\rho^0_h\|_{L^1(\Omega)}  \left(\int_0^T \|\rho_{h,k}^+\|^2_{L^2(\Omega)}\dt\right)^\frac{1}{2} \left(\int_0^T\|\nabla\mathcal{SZ}_h\rho\|_{L^2(\Omega)}^2\dt\right)^\frac{1}{2}\to 0.
\end{align*}
For the second term, we apply \eqref{strong-convergence-rho-II} to show
$$
\int_0^T\|\nabla K*[\rho_{h,k}^-]_T-K*\rho_T\|^2_{L^\infty(\Omega)}\dt \le \|\nabla K\|^2_{L^\infty(\R^d)} \int_0^T \|[\rho_{h,k}^-]_T-\rho_T\|^2_{L^1(\Omega)}\dt \to 0,
$$
which implies on recalling \eqref{weak-convergence-rho}  that
$$
\int_0^T(\rho_{h,k}^+\nabla K*[\rho_{h,k}^-]_T,\nabla\mathcal{SZ}_h\bar\rho)\dt\to \int_0^T(\rho\nabla K*\rho_T,\nabla\mathcal{SZ}_h\bar\rho)\,\dt.
$$
Therefore,
$$
\int_0^T(\rho_{h,k}^+\nabla\mathcal{I}_h( K*[\rho_{h,k}^-]_T),\nabla\mathcal{SZ}_h\bar\rho)\dt\to \int_0^T(\rho\nabla K*\rho_T,\nabla\mathcal{SZ}_h\bar\rho)\,\dt.
$$
as $(h,k)\to (0,0)$.
\end{itemize}

The continuous assimilation of the initial datum is ensured by the compact embedding $H^1(0,T; (H^1(\Omega))')$ into $C([0,T]; (H^1(\Omega))')$ and \eqref{initial-convergences}. Moreover, one can show that $\rho(t)\to \rho^0$ in $L^p(\Omega)$ as $t\to 0$. For more details, see \cite[pp. 1627]{Bertozzi_Slepcev_2010}.

Since $\bar\rho\in L^2(0,T, W^{1,\infty}(\Omega))$ is dense in $L^2(0,T; H^1(\Omega))$, we have found $\rho: \overline{Q}\to [0,\infty)$ such that
$$\rho\in L^\infty(0,T; L^2(\Omega))\cap H^1(0,T; (H^1(\Omega))'), $$
$$\int_\Omega \rho(t)\dx=\int_\Omega \rho_0\dx\quad\mbox{ for all}\quad t\in[0,T], $$
and
\begin{equation}\label{PDE-L2H-1-trunc}
\left\{
\begin{array}{rcccl}
\partial_t\rho-\Delta  A(\rho_T) +\nabla\cdot (\rho\nabla K*\rho_T)&=&0&\mbox{ in }& L^2(0,T; (H^1(\Omega))'),
\\
\rho(0)&=&\rho_0&\mbox{ in }& (H^1(\Omega))'.
\end{array}
\right.
\end{equation}
To complete with the proof of Theorem \ref{Thm:Main} we show the equivalence of problems \eqref{PDE-L2H-1} and \eqref{PDE-L2H-1-trunc}.
\begin{lemma} Problems \eqref{PDE-L2H-1} and \eqref{PDE-L2H-1-trunc} are equivalent.
\end{lemma}
\begin{proof} At this point the only thing we need to show is that $\rho$ defined by \eqref{PDE-L2H-1-trunc} satisfies $\rho\le B_{L^\infty}$ in $Q$. Indeed, define  $\rho_{\rm aux}= e^{ t\|\Delta K\|_{L^{\infty}(\mathds{R}^d)}\|\rho_{0}\|_{L^1(\Omega)}}\|\rho_0\|_{L^\infty(\Omega)}$ for $t\in[0,T]$, and observe that, by \eqref{Kxn},
\begin{equation}\label{lm7.1-lab1}
(\partial_t \rho_{\rm aux}, \bar\rho)-(\rho_{\rm aux}\nabla K*\rho_T,\nabla\bar\rho)\le0
\end{equation}
holds for all $\bar\rho\in H^1(\Omega)$ with $\bar\rho\le0$. Substracting \eqref{PDE-L2H-1-trunc} from \eqref{lm7.1-lab1}, and testing the resulting equation against $\widetilde\rho=(\rho_{\rm aux}-\rho)_-\in H^1(\Omega)$ yields
$$
\frac{1}{2}\frac{d}{dt}\|\widetilde\rho\|^2_{L^2(\Omega)}-(\nabla A_T(\rho), \nabla\widetilde\rho)-(\widetilde\rho\nabla K*\rho_T, \nabla\widetilde\rho)\le 0,
$$
or equivalently,
$$
\frac{1}{2}\frac{d}{dt}\|\widetilde\rho\|^2_{L^2(\Omega)}+\|(A^{'}_T)^\frac{1}{2}(\rho)\nabla\widetilde\rho\|^2_{L^2(\Omega)}-(\widetilde\rho\nabla K*\rho_T, \nabla\widetilde\rho)\le 0.
$$
It follows again from \eqref{Kxn} and integration by parts that
$$
\frac{d}{dt}\|\widetilde\rho\|^2_{L^2(\Omega)}\le \|\Delta K\|_{L^\infty(\R^d)} \|\rho_0\|_{L^1(\Omega)} \|\widetilde\rho\|^2_{L^2(\Omega)}
$$
and so $\rho\le \rho_{\rm aux}\le B_{L^\infty}$ by Grönwall's lemma. Therefore, $\rho_T=\rho$.
\end{proof}

\section{Simulation of aggregation phenomena}

In this section we illustrate how scheme \eqref{Scheme} can be used to approximate the unique 
weak solution to \eqref{PDE} with \eqref{BC}-\eqref{IC}. Moreover, we compare our numerical solution 
to that computed in \cite[Sect. 3.4, Ex. 8]{Carrillo_Chertock_Huang_2015}.

\subsection{Computational performance} At this point we shall make two comments regarding scheme~ \eqref{Scheme}. 
Firstly, we need not use the truncating operator $[\cdot]_T$ to compute $A([\rho^{n+1}_h]_T)$ and $K*[\rho_h^n]_T$ because  the discrete approximations are non-negative and the unique weak solution being approximated is not expected to blow up. Furthermore the convolution term $K*\rho^{n}_h$ cannot be exactly computed at the nodes in order to construct its nodal interpolation, so a quadrature formula must be utilized on simplexes. 

Then, our numerical method remains as:  Given $\rho^n_h\in D_h$, compute $\rho^{n+1}_h\in D_h$ satisfying
\begin{equation}\label{Scheme_II}
(\delta_t\rho^{n+1}_h,\bar\rho_h)_h+ h^\gamma (\nabla\rho^{n+1}_h,\nabla\bar\rho_h)+(\nabla\mathcal{I}_h  A(\rho^{n+1}_h),\nabla\bar\rho_h)-(\rho^{n+1}_h\nabla \mathcal{I}_h( \mathcal{Q}_1(K*\rho^n_h))),\nabla\bar\rho_h)=0,
\end{equation}
where $\mathcal{Q}_1$ is the midpoint quadrature formula. The term $(\delta_t\rho^{n+1}_h,\bar\rho_h)_h$ can be computed by using a closed-nodal quadrature formula, and the term $\mathcal{I}_h(\mathcal{Q}_1(K*\rho^n_h))$ can be rewritten as follows. Let $\a\in \mathcal{N}_h$, then
\begin{equation}\label{Quadrature}
Q_1(K*\rho^n_h)(\a)=\int_\Omega \widetilde{\mathcal{P}}_h(K(\a-\y) \rho^{n}_h(\y)) d\y=\sum_{E\in\mathcal{E}_h} K(\a-\b_E) |E|,
\end{equation}
where $\widetilde{\mathcal{P}}_h$ is a piecewise constant interpolation taking its value on each $E\in\mathcal{E}_h$ at the barycenter $\b_E$.

We see no obstacle to analyzing algorithm \eqref{Scheme_II} using \eqref{Quadrature} and the truncating $[\cdot]$ in $A(\cdot)$ as well, but we did not consider such a formulation in our analysis because it is tedious.

Scheme \eqref{Scheme}, and its modification \eqref{Scheme_II}, require the solution of  nonlinear algebraic systems at each time step, which can be approximately solved using fixed-point iterations. In doing so, we first observe that $A(\rho^{n+1}_h)=\mathcal{D}(\rho^{n+1}_h)\nabla \rho^{n+1}_h$, where $\mathcal{D}(\rho^{n+1}_h)$ is a piecewise constant, $d\times d$ diagonal matrix function over the mesh $\mathcal{T}_h$ constructed as follows. Let $E\in\mathcal{E}_h$ and consider $\tilde E\subset E$ to be a right simplex (see Figure~\ref{fig:right-angle-triangle}) with vertices
$\{ \widetilde{\boldsymbol{a}}_j\}_{j=0,\cdots,d}$ with $\widetilde{\boldsymbol{a}}_0$ supporting
the right angle. Then
$$
[\mathcal{D}(\bar\rho_{h})|_E]_{jj}=
\left\{
\begin{array}{cl}
  \dfrac{A(\bar\rho_{h}(\boldsymbol{a}_j))-A(\bar\rho_{h}(\boldsymbol{a}_0))}{\bar\rho_{h}(\boldsymbol{a}_j)-\bar\rho_{h}(\boldsymbol{a}_0)}
  & \hbox{if }  \bar\rho_{h}(\boldsymbol{a}_j)-\bar\rho_{h}(\boldsymbol{a}_0)\not=0,
  \\
  0 & \hbox{if }  \bar\rho_{h}(\boldsymbol{a}_j)-\bar\rho_{h}(\boldsymbol{a}_0)=0.
\end{array}
\right.
$$
Particularly, we choose $\widetilde{\boldsymbol{a}}_0$  to be the incenter of $E$. Thus, we take $\widetilde{\boldsymbol{a}}_i=\widetilde{\boldsymbol{a}}_0+ \frac{r_E}{2} \e_i$, where $r_E$ is the inradius of the inscribed ball.

\begin{figure}
  \centering
  \includegraphics[width=0.25\linewidth]{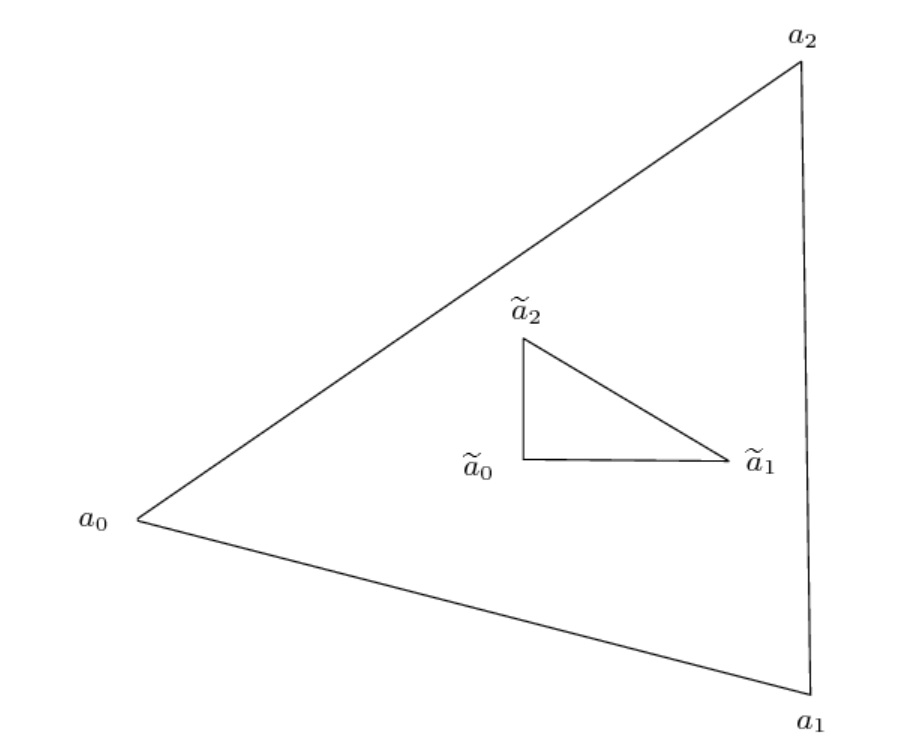}
  \caption{Interior right triangle used for computing $\mathcal{D}(n_{h})$}
  \label{fig:right-angle-triangle}
\end{figure}

We linearize as follows. For $i=0$, select
$\rho^{n+1}_{h,i}=\rho^{n}_h$, then compute $\rho^{n+1}_{h,i+1}$ using
in~(\ref{Scheme_II}) the expression
\begin{align*}
(\rho^{n+1}_{h,i+1},\bar\rho_h)_h+ h^\gamma (\nabla\rho^{n+1}_{h,i+1},\nabla\bar\rho_h)+&(\mathcal{D}(\rho^{n+1}_{h,i})\nabla \rho^{n+1}_{h,i+1},\nabla\bar\rho_h)
\\
&-(\rho^{n+1}_{h,i+1}\nabla \mathcal{I}_h( \mathcal{Q}_1(K*\rho^n_h))) =(\rho^{n}_{h},\bar\rho_h)_h.
\end{align*}
As a stopping criterion for the iterations, we choose  $\|\rho^{n+1}_{h,i+1}-\rho^{n+1}_{h,i}\|_{L^2(\Omega)}<\emph{tol}$, with $tol$ being the prescribed tolerance.

Finally, the computation of \eqref{Quadrature} for each $\a\in\mathcal{N}_h$ constitutes the bottleneck in running scheme \eqref{Scheme_II}. To make it possible in an acceptable amount of time, a parallel procedure on a high-performance cluster can be invoked since all the nodes $\a\in\mathcal{N}_h$ are independent of each other.

\subsection{A numerical experiment} As the domain we take the square $\overline\Omega = [-4,4]^2\subset\R^2$.  The evolution starts from the initial datum $\rho^0=\frac{1}{4}\,\chi_{[-3,3]^2}$ being a rescaled characteristic function supported in the square $[-3,3]^2$, which is shown in Figure \eqref{fig:initial_datum}.  The local repulsion term is chosen as $A(\rho)=\frac{\nu}{m}\rho^m$ with $\nu=0.1$ and $m=3$, and the kernel is set as $K(\boldsymbol{x})=\exp(-|\boldsymbol{x}|^2)/\pi$.
\begin{figure}
   \centering
   \includegraphics[width=0.33\linewidth, height=0.3\linewidth]{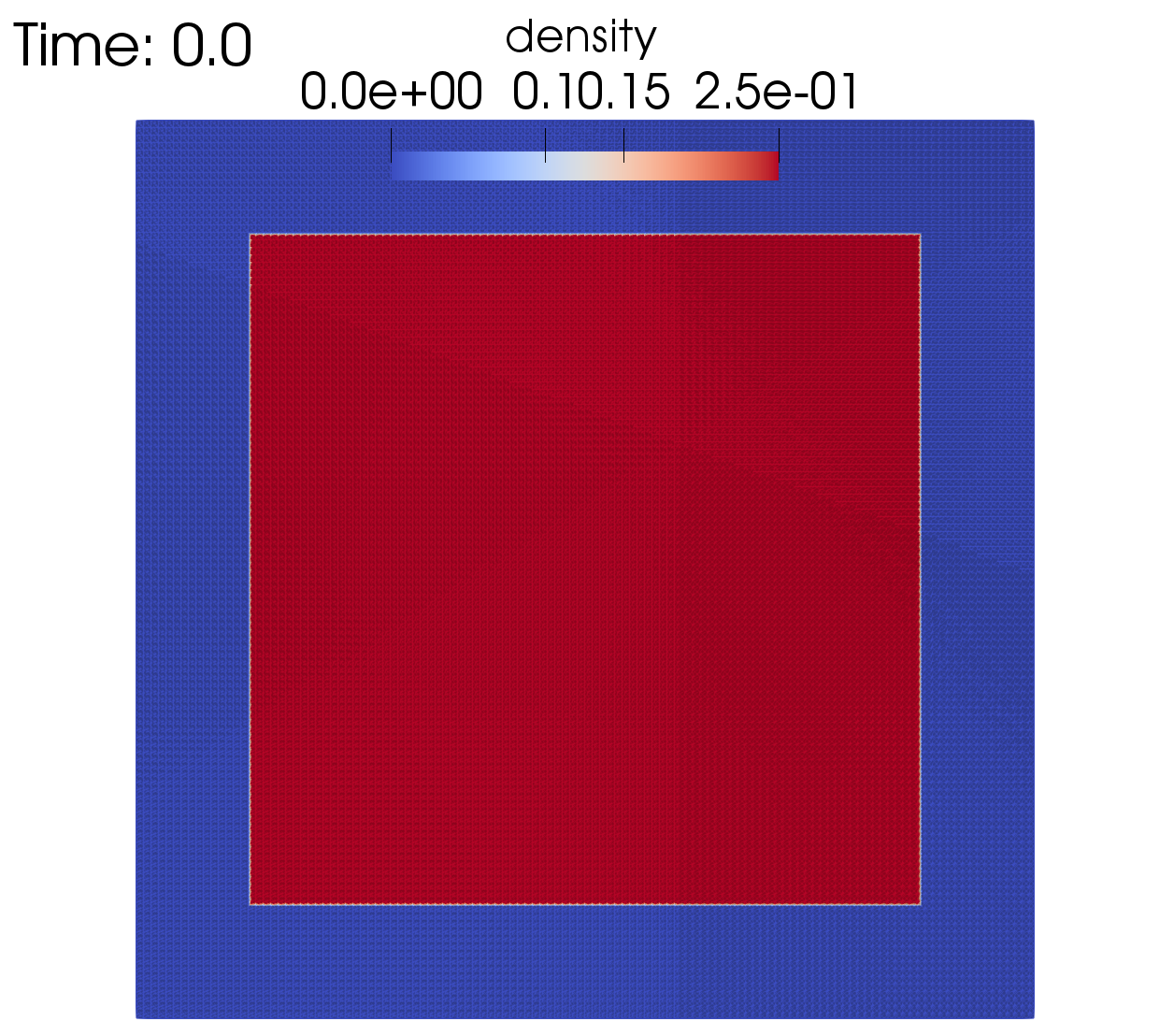}
  \caption{The initial condition $\rho^0_h$.}
  \label{fig:initial_datum}
\end{figure}
From an $N_{\rm square}\times N_{\rm square}$ uniform grid, obtained by dividing $\Omega$ into macroelements consisting of squares, we construct the mesh $\mathcal{T}_h$ by splitting each macroelement  into 14 acute triangles as indicated in Figure~\ref{fig:acute_macrolement}.
\begin{figure}
  \centering
  \includegraphics[width=0.33\linewidth, height=0.3\linewidth]{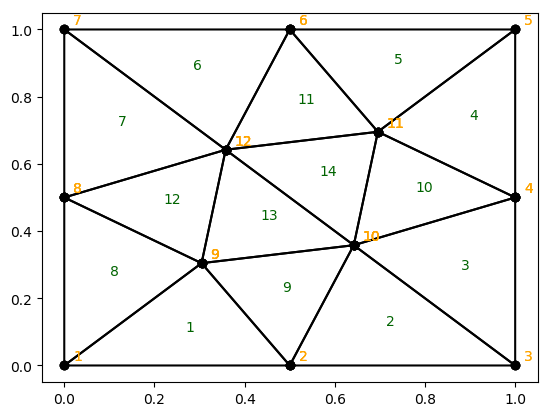}
  \caption{Reference macrolement, composed of 14 acute triangles}
  \label{fig:acute_macrolement}
\end{figure}
This way, for $N_{\rm square}=120$, we define a mesh consisting of $201600$ acute triangles and $101281$ vertices with $h=8/(2 N_\mathbf{x})\simeq 0.033$. For the time discretization, we carry out $1500$ iterations with time step $k=10^{-1}$. Moreover we select $\gamma=0.99$ to be as large as possible in orden to reduce the impact of the stabilizing term, which has a smoothing effect on the dynamics of aggregation phenomena. It should be noticed that $(h,k)$ do not fulfill \eqref{Restriction_II-global} what makes us believe that such a restriction is superfluous. We also performed some numerical tests with $k=10^{-2}$, obtaining quite similar results, which are omitted for brevity.

Specifically, we run our test on a machine with $16$ Intel Xeon $E5 2670$ processors ($2,6$ GHz, $8$-core), in a distributed memory architecture; thus using a total of 256 parallel threads. The MPI library on the FreeFem++ PDE solver \cite{Hecht_2012} was selected as a software framework.

Using the above-described parallel computing environment for the computation of \eqref{Quadrature} for each $\a\in\mathcal{N}_h$, the converged solution is obtained by about $3$ iterations with tolerance $tol=10^{-3}$ in the $L^2(\Omega)$-norm. To be more precise, the average number of iterations is $2.81$, with minimum and maximum equal to $2$ and $11$, respectively. So, each time step takes an average time of $88.85$ seconds, of which $35.95$ seconds (on average) are due to the parallel computation of \eqref{Quadrature}. The remaining time is occupied in solving the associated linear system, which spends  $3.06$ seconds (on average) for each iteration, and data I/O. 

Figure~\ref{fig:evolution} shows how the initial state changes into four peaks that are aggregated into a single component until reaching a final steady state. This result is in good agreement with that in \cite[Sect. 3.4, Ex. 8]{Carrillo_Chertock_Huang_2015}. The dynamics regarding the $\|\cdot\|_{L^\infty(\Omega)}$- and $\|\cdot\|_{L^1(\Omega)}$-norms is reported in Figure \ref{fig:Linf-L1-norm}. The $\|\cdot\|_{L^\infty(\Omega)}$- norm approaches the value $16$ as of $t=14$, while the $\|\cdot\|_{L^1(\Omega)}$-norm takes values around $8.92$, which is comparable to $\|\rho^0\|_{L^1(\Omega)}=9$.
\begin{figure}
\centering
\subfigure[Time 2.5]{
    \includegraphics[scale=0.1]{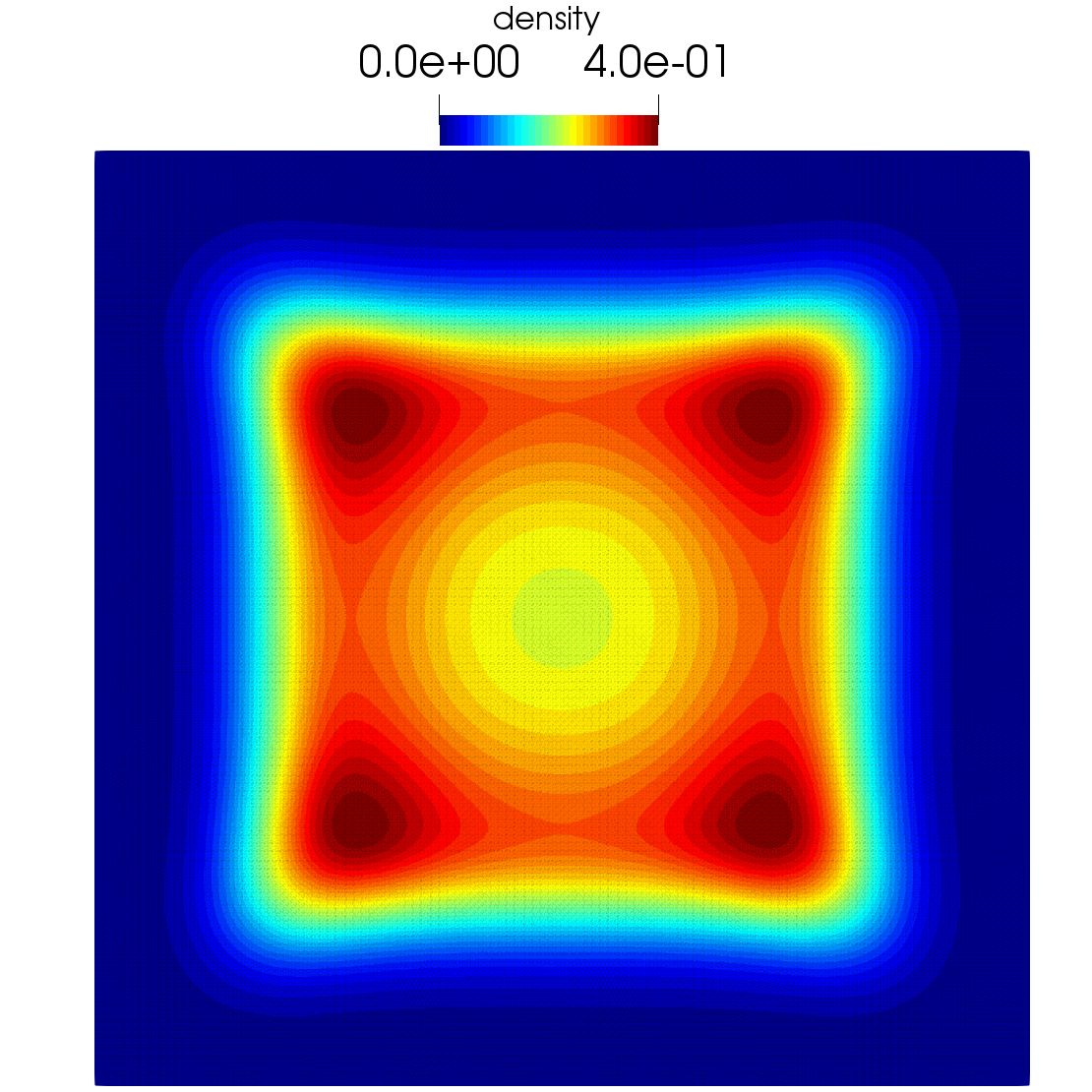}}
    \hspace{0.5cm}
\subfigure[Time 5.0]{
      \includegraphics[scale=0.1]{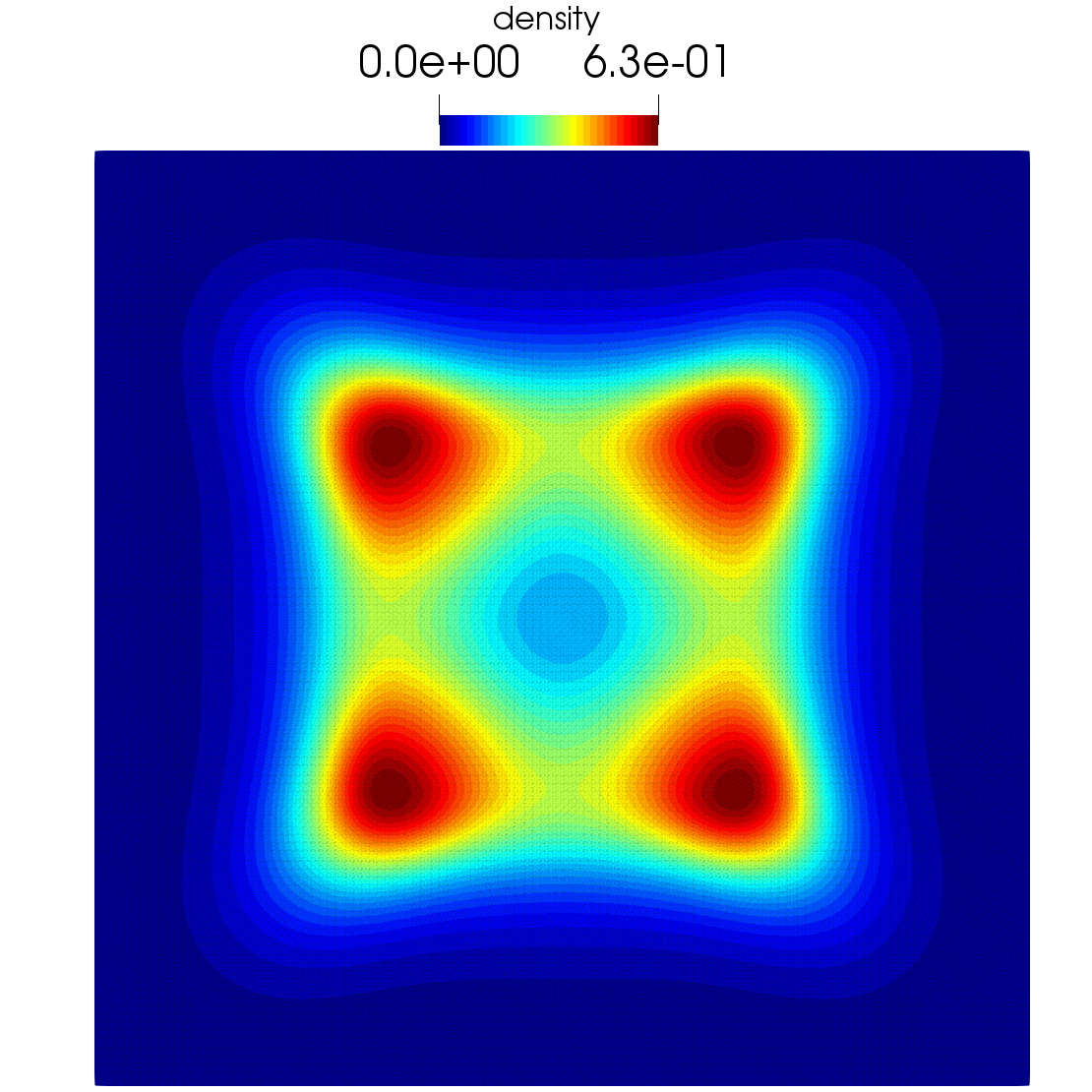}}
    \hspace{0.5cm}      
\subfigure[Time 7.5]{
    \includegraphics[scale=0.1]{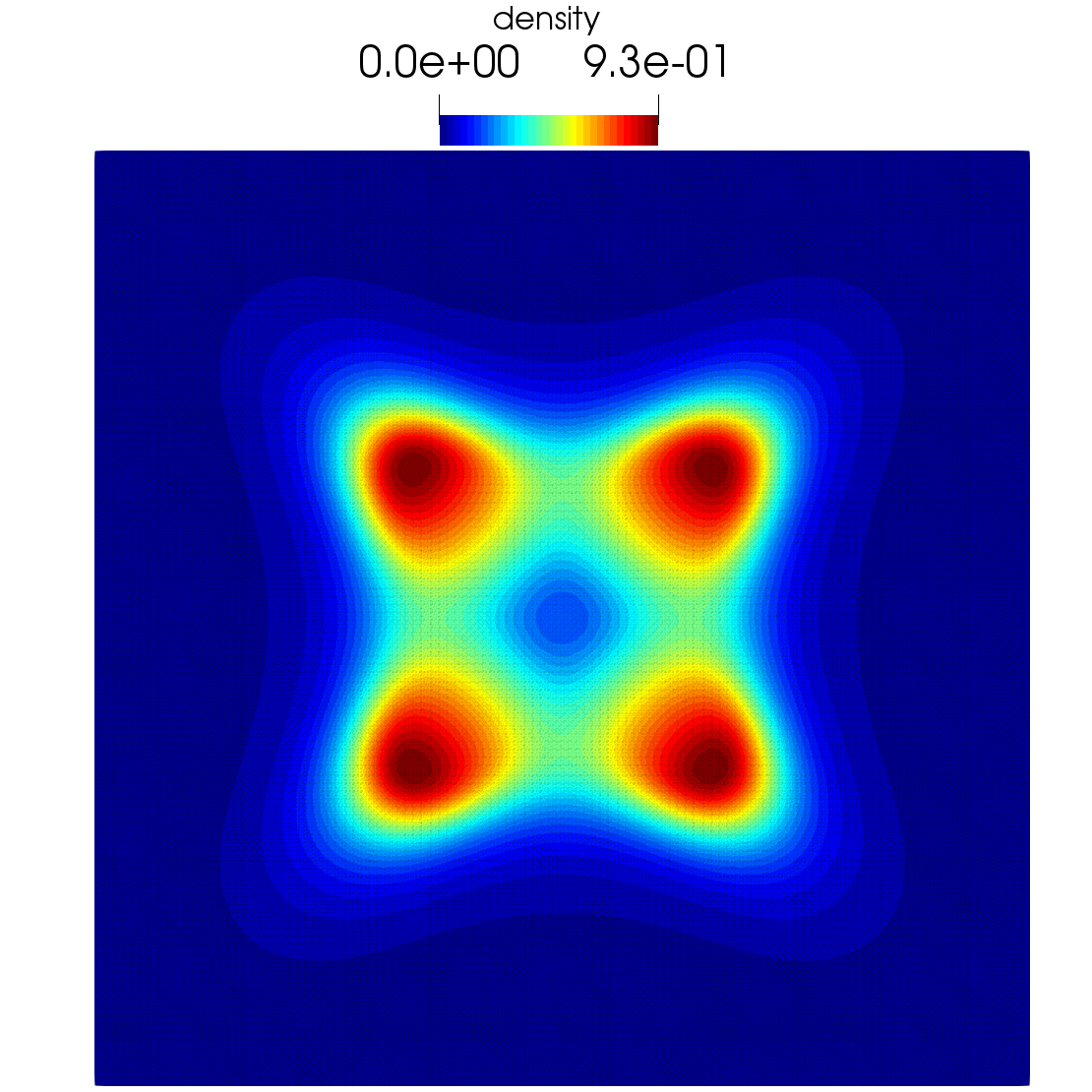}}
    \hspace{0.5cm}
\subfigure[Time 10.0]{
      \includegraphics[scale=0.1]{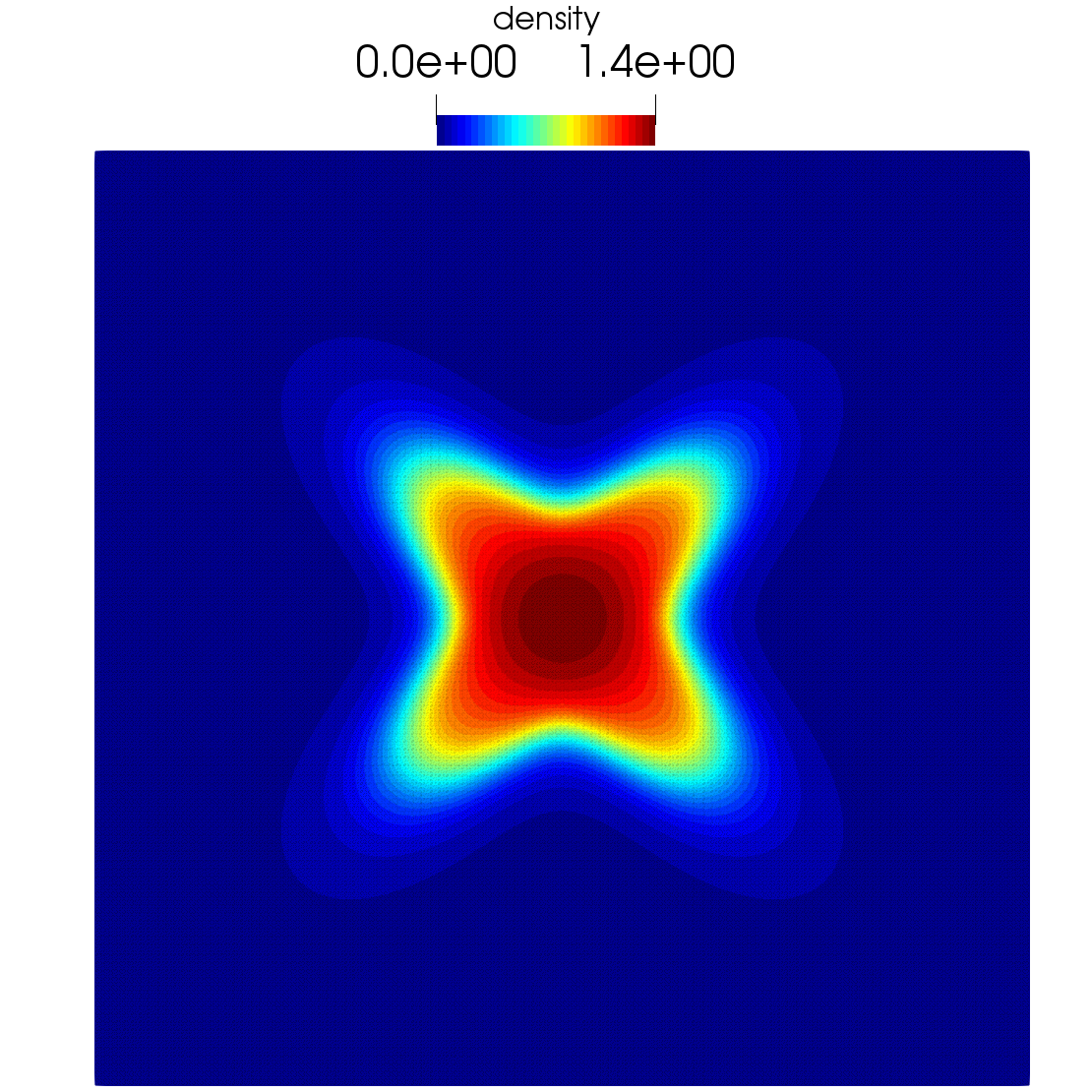}}
    \hspace{0.5cm}
\subfigure[Time 12.5]{
    \includegraphics[scale=0.1]{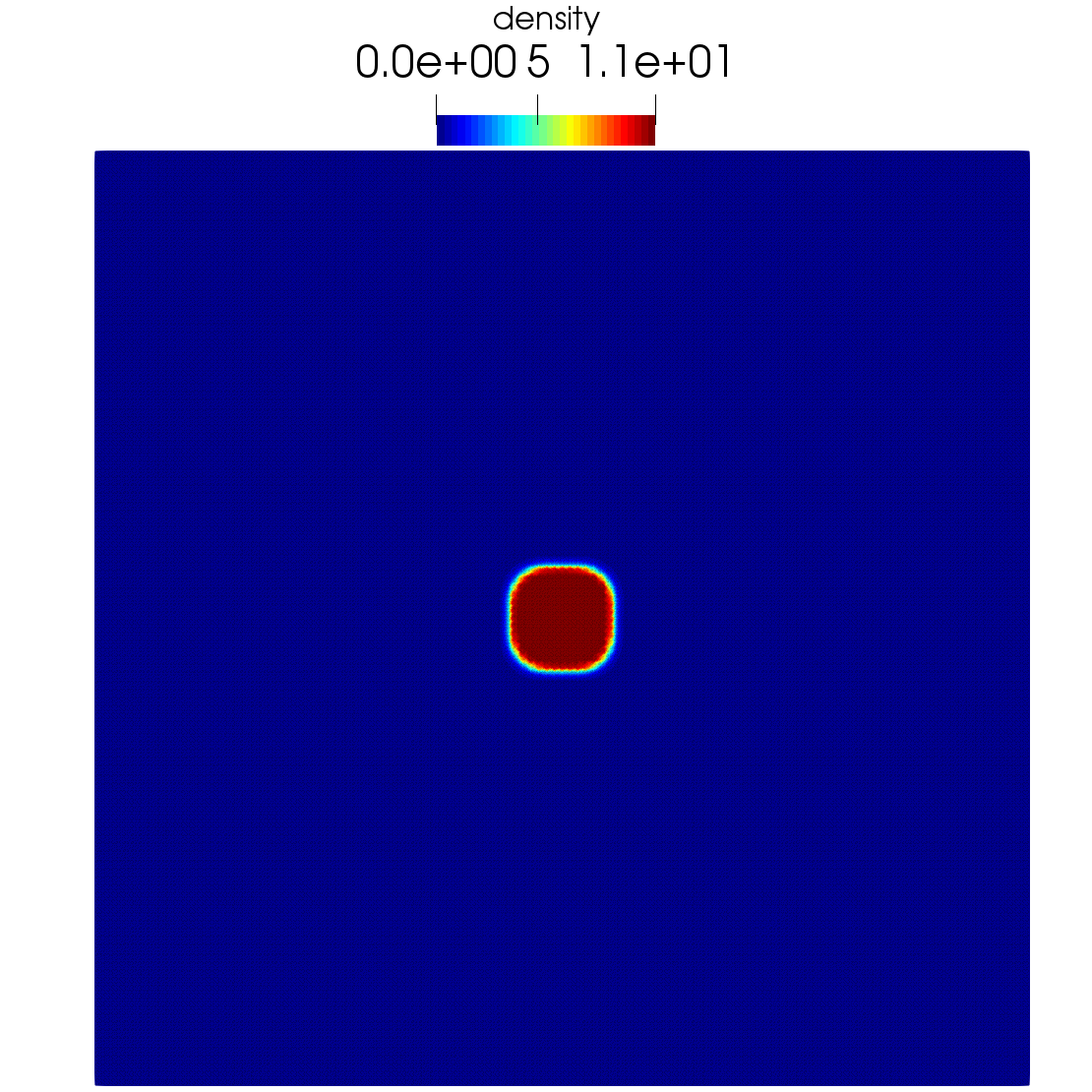}}
    \hspace{0.5cm}
\subfigure[Time 15.0]{
      \includegraphics[scale=0.1]{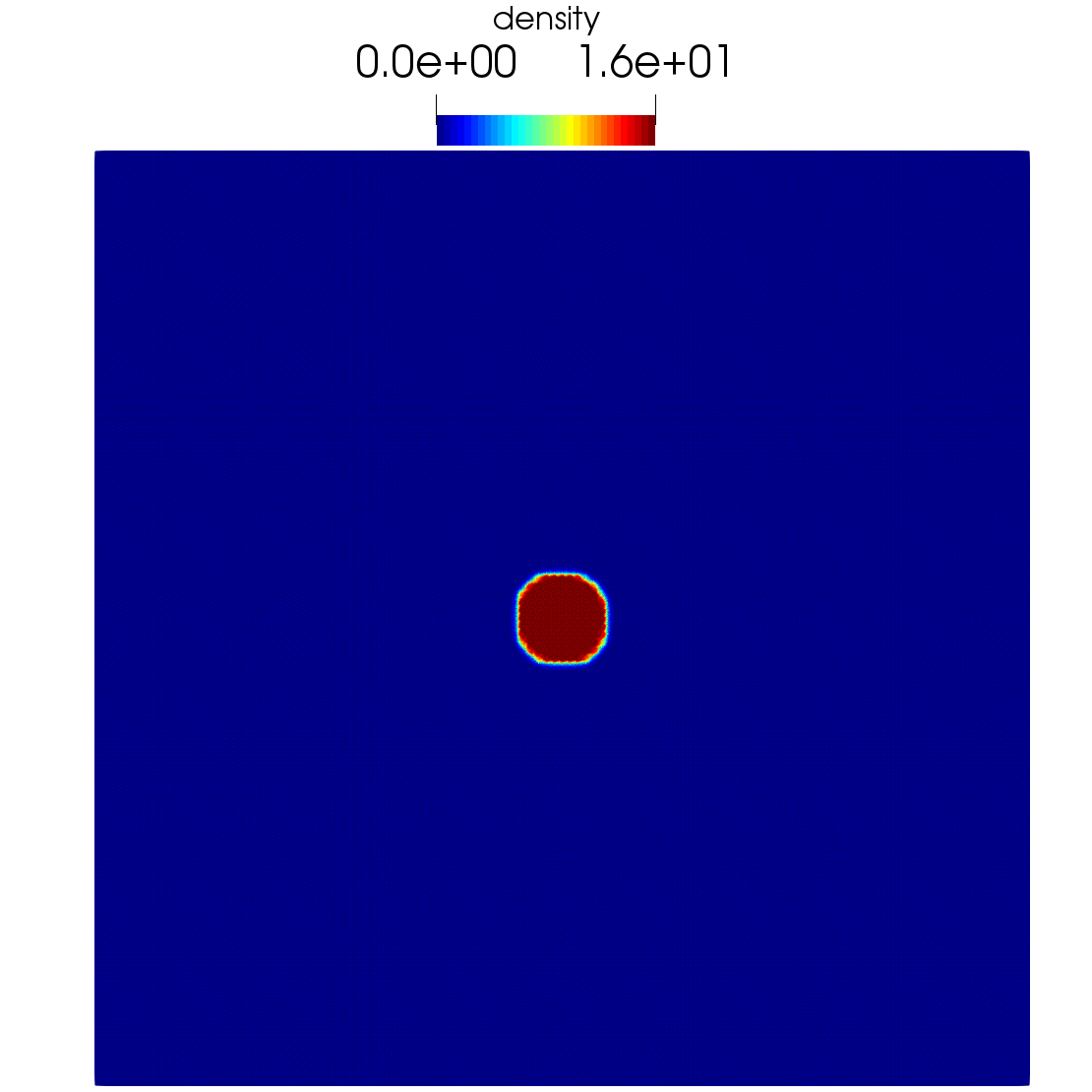}}
  \caption{Evolution of the computed solution at times $t=2.5, 5, 7.5,10, 12.5, 15$}
  \label{fig:evolution}
\end{figure}

\begin{figure}
  \centering
  \begin{tabular}{ccc}
    \includegraphics[width=0.5\linewidth, height=0.3\linewidth]{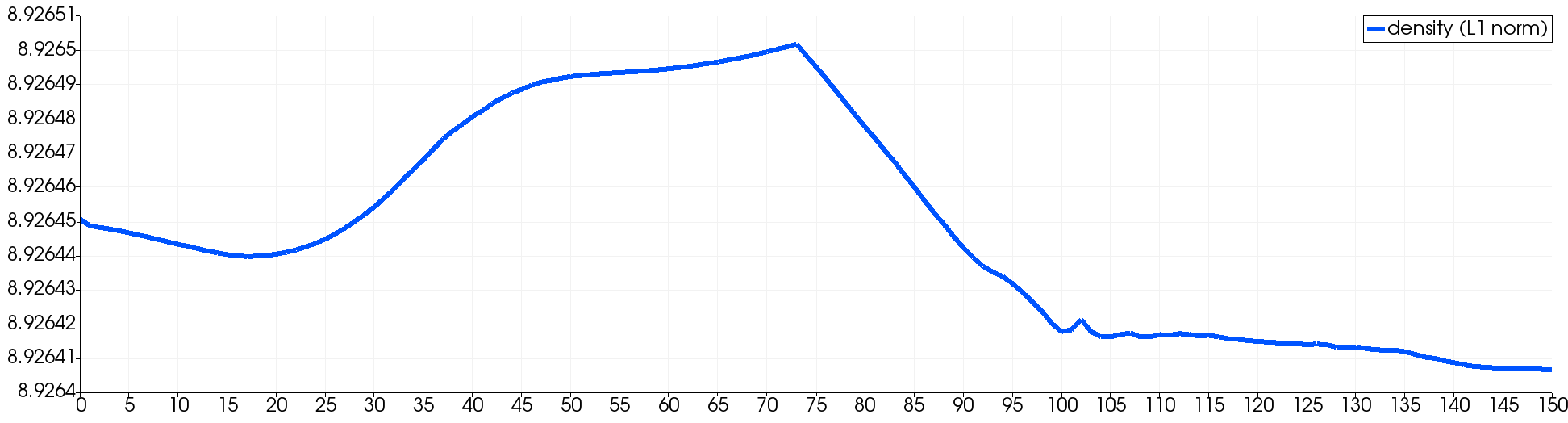}
    &
      \includegraphics[width=0.5\linewidth, height=0.3\linewidth]{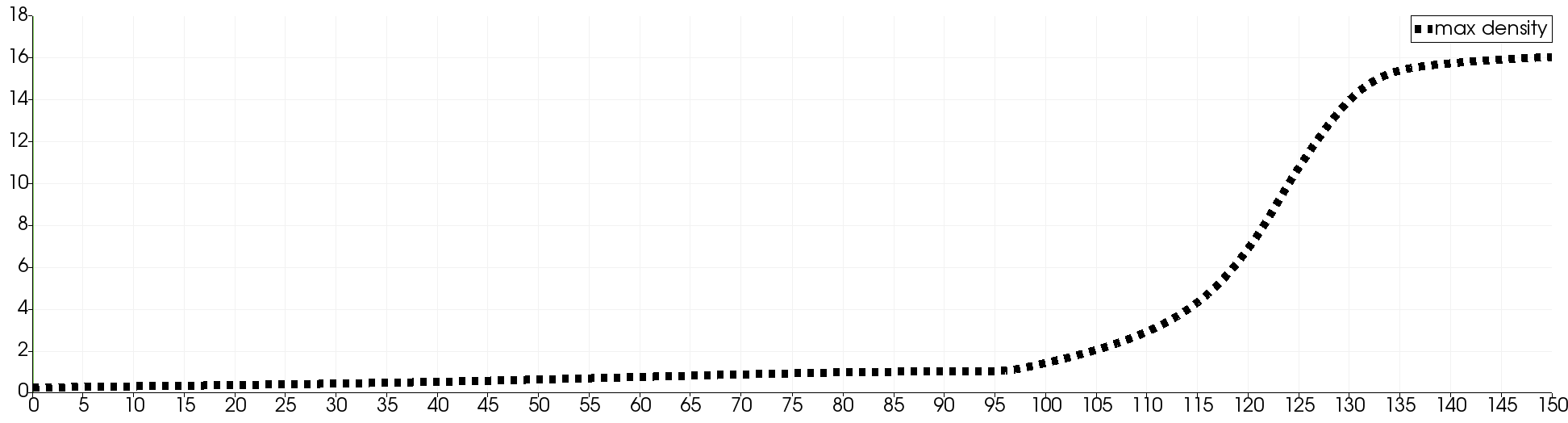}
\end{tabular}
%   \begin{tabular}{ccc}
%     \includegraphics[width=0.75\linewidth, height=0.3\linewidth]{figures/L1_norm.png}
% \\
%     \includegraphics[width=0.75\linewidth, height=0.3\linewidth]{figures/Linf_norm.png}
% \end{tabular}
  \caption{Plot of the $\|\cdot\|_{L^1(\Omega)}$-norm (left) and the $\|\cdot\|_{L^\infty(\Omega)}$-norm (right) for the computed density.}
  \label{fig:Linf-L1-norm}
\end{figure}

\end{document}